\documentclass[a4paper,11pt]{article}

\usepackage{enumitem} 
\usepackage{amssymb}
\usepackage{hyperref}
\usepackage{amsfonts}
\usepackage{amsthm}
\usepackage{amsmath}
\usepackage{amscd}
\usepackage[latin2]{inputenc}
\usepackage{t1enc}
\usepackage[mathscr]{eucal}
\usepackage{indentfirst}
\usepackage{graphicx}
\usepackage{graphics}
\usepackage{pict2e}
\usepackage{epic}
\numberwithin{equation}{section}
\usepackage[margin=3cm]{geometry}
\usepackage{epstopdf} 
\usepackage{comment} 

\usepackage{xcolor}

\theoremstyle{plain}
\newtheorem{thm}{Theorem}[section]
\newtheorem{lem}[thm]{Lemma}
\newtheorem{cor}[thm]{Corollary}
\newtheorem{prop}[thm]{Proposition}

\theoremstyle{definition}
\newtheorem{defn}[thm]{Definition}

\newtheorem{rem}[thm]{Remark}
\newtheorem{?}[thm]{Problem}
\newtheorem{exa}[thm]{Example}

\theoremstyle{definition}
\newtheorem*{nt*}{Notation}

\theoremstyle{plain}
\newtheorem{asu}{Assumption}[section]

\newcommand{\inner}[2]{\left\langle {#1}, {#2} \right\rangle}
\newcommand{\norm}[1]{\left\Vert#1\right\Vert}

\newcommand{\calh}{\mathcal H}

\newcommand {\R} {\mathbb R}

\newcommand {\be} {\beta}
\newcommand {\om} {\omega}
\newcommand {\ome} {\Omega}
\newcommand {\de} {\Delta}
\newcommand {\al} {\alpha}
\newcommand {\ta} {\theta}
\newcommand {\ld} {\lambda}

\begin{document}
\title{A third-order dynamical system for solving generalized monotone equations}

\author{Pham Viet Hai
\thanks{Faculty of Mathematics and Informatics, Hanoi University of Science and Technology, Khoa Toan-Tin, Dai hoc Bach khoa Hanoi, 1 Dai Co Viet, Hanoi, Vietnam. Email: hai.phamviet@hust.edu.vn}
\and
Phan Tu Vuong\footnote{ Corresponding Author}
\thanks{Mathematical Sciences, University of Southampton, SO17 1BJ, Southampton, UK. Email: t.v.phan@soton.ac.uk}
}

\maketitle


\begin{abstract}
We propose a third-order dynamical system for solving a nonlinear equation in Hilbert spaces where the operator is cocoercive with respect to the zero set. Under mild conditions on the parameters, we establish the existence and uniqueness of the generated trajectories as well as their asymptotic convergence to a solution of the equation. When the operator is strongly monotone with respect to the zero set, we establish an exponential convergence rate of $e^{-3t}$, which is faster than the known results of second-order dynamical systems. In addition, we discuss the applications of the proposed dynamical system to several splitting monotone inclusion problems. In particular, for convex optimization problems, the proposed dynamical system provides a fast convergence rate of {\bf $\mathcal{O}(\frac{1}{t^3})$} for the objective values. Examples and some time rescaling transformations are discussed to show the effectiveness of considering third-order ODEs over first- and second-order ODEs.
\end{abstract}



\section{Introduction}
Let $\calh$ be a real Hilbert space with the inner product $\left<\cdot, \cdot \right>$ and the induced norm $\|\cdot\|$. Let $U: \calh \to \calh$ be a continuous operator.
We are interested in the following equation:
\begin{gather}\label{pb-zer}
\text{Find $x_*$ such that $U(x_*)=0$.}    
\end{gather}
The symbol $Z(U)$ stands for the zero set of the operator $U$, which is assumed to be nonempty. The equation \eqref{pb-zer} looks simple, but it is general enough to cover numerous problems in optimization and variational analysis, such as convex optimization problems, variational inequalities,  monotone inclusions, fixed-point problems, saddle-point problems. Let us recall some particular examples below:\\

{\bf Optimization problem:} Let $f: \calh \to \R$ be a convex and differentiable function. Then solving the optimization problem 
\begin{gather}\label{Op}
\min_{x \in \calh} f(x) 
\end{gather}
is equivalent to solving Problem \eqref{pb-zer} with $U:=\nabla f$, where the notation $\nabla f$ denotes the gradient of $f$.\\

{\bf Monotone inclusion:}
Another important problem is to find a zero of the sum of two monotone operators
\begin{gather}\label{proA+B}
    \text{Find $x_*\in \calh$ such that}\quad 0\in A(x_*)+B(x_*),
\end{gather}
where $A:\calh\rightarrow\calh$ is a monotone single-valued operator and 
$B:\calh \rightrightarrows\calh$ is a maximal monotone set-valued operator defined on $\calh$.
Let $I$ be the identity operator and $J_B:=(I+B)^{-1}$ be the resolvent of the operator $B$. Then the monotone inclusion \eqref{proA+B} is a particular case of Problem \eqref{pb-zer} with
$$
U:= I-J_{\gamma B} (I-\gamma A),
$$
for some $\gamma >0$. \\

{\bf Variational inequality:} A crucial special case of Problem \eqref{proA+B} is the following variational inequality problem
\begin{gather}\label{proA+N1}
    \text{Find $x_*\in\calh$ such that}\quad 0\in A(x_*) + N_\Omega(x_*),
\end{gather}
where $\Omega$ is a nonempty closed convex subset of $\calh$ and $N_\Omega(x_*)$ is the normal cone of $\Omega$ at $x_*.$ 

{\bf Fixed-point problem:} Let $F:\calh \to \calh $ be an operator. The fixed-point problem:
\begin{gather*}
\text{Find $x\in \calh$ such that $x = F(x)$}
\end{gather*}
is equivalent to Problem \eqref{pb-zer} with 
$U:= I-F$. \\

There are recent emerging research directions that use ordinary differential equations (ODEs) to investigate the zero set of optimization problems \cite{ zbMATH06857154, zbMATH06622014, zbMATH06305478, zbMATH06655069}, variational inequalities \cite{zbMATH07212695, zbMATH06963496, zbMATH07389165}, monotone inclusions \cite{zbMATH07818477, BSV}
and equilibrium problems \cite{zbMATH07570470, zbMATH07211750} and many others. 
We refer the reader to \cite{zbMATH06309769, zbMATH07408098, zbMATH07573368, zbMATH06622014, zbMATH06305478,zbMATH06588608, zbMATH06655069} and references therein for more examples. It is well understood in the literature that the gradient descent algorithm for solving optimization is a discrete version of the gradient flow (first-order ODE),  which provides a convergence rate of $\mathcal{O}(\frac{1}{t})$. Meanwhile, the second-order ODE, which associates with the Nesterov's algorithm,  exhibits a fast convergence rate of $\mathcal{O}(\frac{1}{t^2})$. Using ODEs does not provide exclusively a novel insight into Nesterov's scheme, it is one of the productive methods to design algorithms with similar performances.\\

    Motivated by the fast convergence rate of the second-order ODE for solving optimization, this approach has been extended extensively to cocoercive equations \cite{zbMATH06588608}, variational inequalities \cite{zbMATH07389165} as well as monotone inclusions \cite{zbMATH06309769, zbMATH06790530}. For solving Problem \eqref{pb-zer}, Bo\c t and Csetnek proposed a second-order ODE where the operator $U$ is cocoercive. They showed the existence and uniqueness of
the generated trajectories as well as their weak convergence to a zero of the operator
$U$. The application to the monotone inclusion \eqref{proA+B} and convex optimization was also discussed \cite{zbMATH06588608}. In particular, when the sum $A+B$ in \eqref{proA+B} is strongly monotone, 
Bo\c t and Csetnek showed in \cite{zbMATH06790530} that the trajectories generated by the second-order ODE converge exponentially to the unique solution with a convergence rate of $e^{-t}$.

\subsection{Motivation for Higher-Order ODEs}
While the theory and results for first- and second-order ODEs are well established as reviewed above, the idea of using a higher-order ODE approach is relatively new. In this paper, we study the equation \eqref{pb-zer} via a {\bf third-order dynamical system}.\\

Third-order dynamical systems for solving the optimization problem \eqref{Op} have been studied recently by Attouch, Chbani and Riahi. In \cite{zbMATH07538428}, they proposed the following ODE 
\begin{equation} \label{thirdDSOpt}
     x'''(t) +\frac{\alpha}{t} x''(t) + \frac{2\alpha-6}{t^2} x'(t) + \nabla f(x(t) + tx'(t)) = 0,
\end{equation}
 called  (TOGES). These authors first reformulated \eqref{thirdDSOpt}  as a second-order dynamical system by temporal scaling techniques. Then the convergence analysis was obtained using Lyapunov's energy function which was well developed for second-order ODEs. More importantly, they showed a convergence rate of $\mathcal{O}(\frac{1}{t^3}$) in the sense that
$$
f(x(t) + t x'(t)) - \inf_{\calh}f 
\le \frac{C}{t^3}
$$
for some constant $C>0$ 
and obtained the convergence of the trajectories towards optimal solutions of \eqref{Op}. 
However, the convergence rate of   $f(x(t))$ in (TOGES) is only  $\mathcal{O}(\frac{1}{t})$, which is not fully satisfactory in the view of fast optimization.  
In \cite{zbMATH07814972}, the authors provided a modification of (TOGES), called (TOGES-V), as follows:
\begin{equation}\label{thirdDSOpt2}
     x'''(t) +\frac{\alpha+7}{t} x''(t) + \frac{5(\alpha+1)}{t^2} x'(t) + \nabla f\left(x(t) + \frac{1}{4}tx'(t)\right) = 0,
\end{equation}
where they finally obtained a fast convergence rate of $\mathcal{O}\left(\frac{1}{t^3}\right)$ for $f(x(t) ) - \inf_{\calh} f$. \\

Recently, we proposed in \cite{HaiVuong24} the following third-order dynamical system for solving monotone inclusion \eqref{proA+B}
\begin{gather}\label{eq3C}
    y'''(t)+\al_2y''(t)+\al_1y'(t)+\al_0(I-J_{\gamma B}(I-\gamma A))(y(t))=0,
\end{gather}
where $\al_2,\al_1,\al_0,\gamma$ are positive coefficients. Under suitable choices of these coefficients, we  obtained the existence and uniqueness of the generated trajectories.  
When $A$ and $B$ satisfy a generalized monotonicity condition and the sum $A+B$ is strongly monotone, we showed that the trajectory $y(t)$ converges exponentially to the unique solution with a fast rate of $e^{-\rho t}$, for some positive $\rho$ which could be chosen so that $\rho>1$ (see \cite[Theorem 3.7]{HaiVuong24} ). This rate is faster than the rate $e^{-t}$ obtained by Bo\c t and Csetnek for second-order ODEs in \cite{zbMATH06790530}. We also discussed the fast exponential convergence results for solving variational inequality problem \eqref{proA+N1} where the operator $A$ is strongly pseudo-monotone and Lipschitz continuous \cite[Section 5]{HaiVuong24}. In \cite{arXiv:2602.11817}, Tran proposed analogous third-order ODEs at comparable speeds for solving generalized inverse mixed variational inequality problems. Comparative examples in \cite{arXiv:2602.11817} emphasize the performance gains achieved by third-order dynamical systems over lower-order ones.

\subsection{Our contributions}
Motivated by the strong evidence of fast convergence results for third-order dynamical systems, we associate with Problem \eqref{pb-zer} an ODE of the following form
\begin{gather}\label{ode3}
\begin{cases}
x'''(t)+\be_2(t)x''(t)+\be_1(t)x'(t)+\be_0(t)U(\phi_x(t))=0,\\
x(t_0)=x_0,x'(t_0)=x_1,x''(t_0)=x_2,
\end{cases}
\end{gather}
where
\begin{gather*}
\phi_x(t):=x(t)+\ld_1(t)x'(t)+\ld_2(t)x''(t).
\end{gather*}
In contrast to the cocoerciveness assumption as required in \cite{zbMATH06588608}, we will only assume that the operator $U$ is cocoercive with respect to the zero set $Z(U)$, which is called {\bf quasi-cocoercive}. This assumption is strictly weaker than the cocoerciveness and useful to tackle the inclusion problem \eqref{proA+B} when $A$ is merely monotone but not cocoercive. Our main contributions are summarized as follows:
\begin{itemize}
\item Existence and uniqueness of the trajectories $x(t)$ generated by \eqref{ode3}. 
\item Weak convergence of the trajectories $x(t)$ to a zero of the operator $U$. Although we extend the weak convergence result from \cite{zbMATH06588608}, its analysis presents challenges, requiring techniques from the theory of third-order ODEs.
\item Fast exponential convergence of $\Vert{}x(t)-x_*\Vert{}^2$ at the rate $\mathcal{O}(e^{-3t})$ for solving strongly monotone (with respect to the zero set) equations. Compared to the rate $e^{-t}$, previously achieved in \cite{zbMATH06790530} for second-order ODEs, and the rate $e^{-2t}$, recently reported in \cite{HaiVuong24} for third-order ODEs, this convergence rate exhibits an improvement. Although a time rescaling $t \mapsto ct$ via $y(t) = x(ct)$ serves as a tool to improve the second-order convergence from $\mathcal{O}(e^{-t})$ to $\mathcal{O}(e^{-ct})$, the mathematical framework of the second-order equation gently constrains its maximum acceleration. In this regard, third-order systems appear to hold a natural advantage, as the time rescaling accelerates their convergence rate from $\mathcal{O}(e^{-3t})$ to $\mathcal{O}(e^{-3ct})$. Remark \ref{rem-compare} details the difference between two systems when subjected to the time rescaling.
\item Weak asymptotic convergence for solving merely splitting monotone inclusions.
\item Fast rate $\mathcal{O}(\frac{1}{t^3})$  for the convex optimization \eqref{Op}, which covers the results in \cite{zbMATH07814972} as a special case. Examples and some time rescaling transformations are discussed to show the effectiveness of considering third-order ODEs over first- and second-order ODEs.
\end{itemize}

The rest of the paper is organized as follows: 
In Section \ref{Preliminaries} we present some preliminary results and lemmas for establishing the convergence analysis. We investigate the existence and uniqueness of the trajectories as well as their asymptotic convergence in Section \ref{main}. In Section \ref{ex}, we establish a fast exponential convergence when the operator $U$ is strongly monotone with respect to the zero set. Applications to several monotone splitting inclusions are presented in Section \ref{Mon}. Finally, Section \ref{Optim} discusses the fast convergence rate of $\mathcal{O}(\frac{1}{t^3})$ for convex optimization problems.

\section{Preliminaries} \label{Preliminaries}
\subsection{Operators}
In this section, we recall some definitions and lemmas which are useful in the sequel. We also present a lemma for establishing the fast rate $\mathcal{O}(\frac{1}{t^3})$ in the following sections.
\begin{defn}
    The single-valued operator $U:\calh\to\calh$ is called
    \begin{enumerate}
        \item \emph{$\om_u$-cocoercive} if $\om_u>0$ and
        \begin{gather*}
            \inner{U(x)-U(y)}{x-y}\geq\om_u\|U(x)-U(y)\|^2\quad\forall x,y\in\calh.
        \end{gather*}
        \item \emph{$\om_u$-{\bf quasi-cocoercive}} if it is cocoercive with respect to the zero set, i.e. 
        \begin{gather*}
            \inner{U(x)}{x-x^*}\geq\om_u\|U(x)\|^2\quad \forall x_*\in Z(U), \forall x\in\calh.
        \end{gather*}
        \item \emph{$L_u$-Lipschitz continuous} if $L_u>0$ and
        \begin{gather*}
            \|U(x)-U(y)\|\leq L_u\|x-y\|\quad\forall x,y\in\calh.
        \end{gather*}
        \item \emph{$\rho$}-strongly monotone if $\rho >0$ and
        \begin{gather*}
            \inner{U(x)-U(y)}{x-y}\geq \rho\|x-y\|^2\quad\forall x,y\in\calh.
        \end{gather*}
        \item \emph{$\rho$}-strongly pseudomonotone if $\rho >0$ and
        \begin{gather*}
            \inner{U(y)}{x-y}\geq 0 \Longrightarrow 
            \inner{U(x)}{x-y}\geq \rho\|x-y\|^2 \quad\forall x,y\in\calh.
        \end{gather*}
    \end{enumerate}
\end{defn}

\begin{rem}
 It is clear that the {\bf quasi-cocoerciveness} is strictly weaker than cocoerciveness. As we will see in Section \ref{Mon}, the forward-backward operator for strongly-pseudo monotone variational inequality and the forward-backward-forward operator for monotone inclusion are quasi-cocoercive, but not necessarily cocoercive.
Also, from the Cauchy-Schwarz inequality that if $T$ is $\om_u$-cocoercive then it is $1/\om_u$-Lipschitz continuous. 
\end{rem}

\begin{defn}
The set-valued operator $B:\calh\rightrightarrows\calh$ is called \emph{monotone} if 
\begin{equation*}
\langle u-v,x-y \rangle \geq  0 \quad \forall \, (x,u), (y,v) \in \text{Gra}B.
\end{equation*}	
 We also say that the operator $B$ is \emph{maximally monotone} if it is monotone and there is no monotone operator whose graph strictly contains the graph Gra$B$.
\end{defn}
If the operator $B$ is maximal monotone, then the \emph{resolvent} of $B$, defined by $J_B: =(I+B)^{-1}$, has full domain. Moreover, it is single-valued and firmly-nonexpansive (i.e.  cocoercive with modulus $1$), see \cite{zbMATH06644703}.

\subsection{Absolutely continuous functions}
Let $\R_{\geq 0}:=[0,\infty)$.
\begin{defn}
A function $h:\R_{\geq 0}\to\calh$ is called \emph{locally absolutely continuous} if for each interval $[t_0,t_1]$ one can find an integrable function $g:[t_0,t_1)\to\calh$ subject to the following equality
\begin{gather*}
    h(t)=h(t_0)+\int\limits_{t_0}^t g(s)\,ds\quad\forall t\in [t_0,t_1].
\end{gather*}
\end{defn}

\begin{rem}
Note that a locally absolutely continuous function is always differentiable almost everywhere and its derivative is identical to its distributional derivative almost everywhere.
\end{rem}

For $m\geq 1$, the space $L^m(X)$ consists of functions that are $m$-power Lebesgue integrable over $X\subseteq\R$ with respect to the Lebesgue measure.

\begin{lem}[{\cite{zbMATH06309769}}]\label{lem-main-1}
Let $u\in L^p(\R_{\geq 0})$, $v\in L^q(\R_{\geq 0})$, where $1\leq p<\infty$ and $1\leq q\leq\infty$. Suppose that $u:\R_{\geq 0}\to\R_{\geq 0}$ is locally absolutely continuous. Then $\lim\limits_{t\to\infty}u(t)=0$ provided that 
\begin{gather*}
u'(t)\leq v(t)\quad\text{for almost every $t\in\R_{\geq 0}$}.
\end{gather*}
\end{lem}

\begin{lem}[{\cite{zbMATH06644703}}]\label{lem-op}
Let $\emptyset\ne\Omega\subseteq\calh$. Assume that the function $\vartheta:\R_{\geq 0}\to\calh$ satisfies conditions (i)-(ii).
\begin{itemize}
    \item [(i)] \quad For every $w\in\Omega$, the limit $\lim\limits_{t\to\infty}\|\vartheta(t)-w\|$ exists.
    \item [(ii)] \quad Every weak sequential cluster point of $\vartheta(\cdot)$ belongs to the set $\Omega$.
\end{itemize}
Then there exists $u\in\Omega$ such that $\vartheta(\cdot)$ converges weakly to $u$ as $t\to\infty$.
\end{lem}

Lemma \ref{lem202404302050} below will be used to prove the convergence at the rate $\mathcal{O}(\frac{1}{t^3})$ for our dynamical system. In the statement, writing $f=\mathcal{O}(g)$ means that $f(t)\leq Cg(t)$ for every $t$, where $C$ is some positive constant.
\begin{lem}\label{lem202404302050}
Let $y:\R\to\calh$ be an absolutely continuous function. If $g:\calh\to\R$ is a convex function satisfying
\begin{gather*}
g\left(y(t)+\xi ty'(t)\right)\leq\frac{M_1}{t^3}\quad\forall t\geq t_0,
\end{gather*}
where $\xi\ne\frac{1}{3},M_1$ are positive constants, then there exist positive constants $M_2,M_3$ such that
\begin{gather*}
g(y(t))\leq\frac{M_2}{t^3}+\frac{M_3}{t^{\frac{1}{\xi}}}\quad\forall t\geq t_0.    
\end{gather*}
Furthermore, if the constant $\xi<\frac{1}{3}$, then
\begin{gather*}
g(y(t))=\mathcal{O}\left(\frac{1}{t^3}\right).
\end{gather*}
\end{lem}
\begin{proof} See the Appendix. 
\end{proof}

The solution of dynamical system \eqref{ode3} is defined as follows.
\begin{defn}
A function $x(\cdot)$ is called a \emph{strong global solution} of equation \eqref{ode3} if conditions (i)-(iii) hold.
\begin{enumerate}
    \item[(i)] The functions $x,x',x'':[t_0,\infty)\to\calh$ are locally absolutely continuous.
    \item[(ii)] The equation $x'''(t)+\be_2(t)x''(t)+\be_1(t)x'(t)+\be_0(t)U(\phi_x(t))=0$ holds for almost every $t\geq t_0$.
    \item[(iii)] $x(t_0)=x_0,x'(t_0)=x_1,x''(t_0)=x_2$.
\end{enumerate}
\end{defn}

\section{A third-order dynamical system 
}\label{main}
We begin this section with two basic observations regarding the dynamical system \eqref{ode3}, where Proposition \ref{prop202511042044} establishes its equivalence to a first-order system and Proposition \ref{prop-202511021525} establishes the existence and uniqueness of its strong global solutions.
\begin{prop}[Equivalent form]\label{prop202511042044}
The dynamical system \eqref{ode3} is equivalent to the first-order system
\begin{gather}\label{202405051657}
y'(t)=T(t,y(t)),
\end{gather}
where $y=(y_1,y_2,y_3)$ and the mapping $T:\R\times\calh\times\calh\times\calh\to\calh\times\calh\times\calh$ is defined by
\begin{align*}
T(t,y)&:=(y_2,y_3,-\be_1(t)y_2-\be_2(t)y_3-\be_0(t)U(\phi_y(t))),\\
\phi_y(t)&:=y_1+\ld_1(t)y_2+\ld_2(t)y_3.    
\end{align*}
\end{prop}
\begin{proof}
The proof is straightforward and it is left to the reader.
\end{proof}

\begin{prop}\label{prop-202511021525}
Consider the dynamical system \eqref{ode3}, where the operator $U:\calh\to\calh$ is $L_u$-Lipschitz continuous and the functions $\be_j:\R_{\geq 0}\to\R_{>0}$ for $j\in\{0,1,2\}$, and $\ld_1,\ld_2:\R_{\geq 0}\to\R_{\geq 0}$ are locally integrable. Then for any $x_0,x_1,x_2\in\calh$, there exists a unique strong global solution to the dynamical system \eqref{ode3}.
\end{prop}
\begin{proof}
See the Appendix.
\end{proof}

\subsection{Weak convergence}
The convergence analysis of the dynamical system \eqref{ode3} will be done by using the Lyapunov functions
\begin{gather}\label{L-func}
    y=\|x-x_*\|^2,\quad z_i=\|x^{(i)}\|^2,
\end{gather}
where $x_*\in Z(U)$. Denote
{\footnotesize
\begin{gather}\label{202303291322}
\begin{cases}
A_2=\frac{(2\om_u-D)\be_1}{\be_0},\\
\\
A_1=\frac{(2\om_u-D)\be_1\be_2}{\be_0}-\frac{\be_0}{D}\ld_1\ld_2-3,\\
\\
A_0=\frac{(2\om_u-D)\be_1^2}{\be_0}-\frac{\be_0}{D}\ld_1^2-2\be_2,
\end{cases}
\begin{cases}
B_1=\frac{(2\om_u-D)\be_2}{\be_0},\\
\\
B_0=\frac{(2\om_u-D)}{\be_0}(\be_2^2-2\be_1)-\frac{\be_0}{D}\ld_2^2,
\end{cases}
C_0=\frac{2\om_u-D}{\be_0}.
\end{gather}
}

Given two functions $f,g:\R_{\geq 0}\to\R$, writing $f\leq g$ ($f\geq g$) means that $f(t)\leq g(t)$ ($f(t)\geq g(t)$ respectively) for every $t\geq 0.$
\begin{asu}\label{asu-main}
The functions $\be_0,\be_1,\be_2:\R_{\geq t_0}\to\R_{>0}$ are locally integrable and satisfy
\begin{gather}
\label{cond1}\be_2''\geq 0\geq\be_2',\\
\label{cond2}\be_1''\geq 0\geq\be_1',\\
\label{cond3}\be_0'\geq 0\geq\be_0''.
\end{gather}
The functions $\ld_1,\ld_2:\R_{\geq t_0}\to\R_{\geq 0}$ are locally integrable and satisfy
\begin{gather}
\label{202404280820}\ld_1'\geq 0,\\
\label{202404280821}\ld_2'\geq 0.
\end{gather}
Suppose a function $D:\R_{\geq t_0}\to\R_{>0}$ exists which satisfies conditions \eqref{202404280817}-\eqref{202404280819}.
\begin{gather}
\label{202404280817}\overline{D}:=\sup\limits_{t\geq t_0}D(t)<2\om_u,\\
\label{202404280818}D^2\geq\frac{\be_0^2\ld_1\ld_2}{\be_1\be_2},\\
\label{202404280819}D''\leq 0\leq D'.
\end{gather}
There exist positive constants $\de_1,\de_2,\de_3$ such that
\begin{gather}
\label{cond4}A_0(t)\geq\de_1\quad\forall t\geq t_0,\\
\label{cond5}B_0(t)\geq\de_2\quad\forall t\geq t_0,\\
\label{cond6}A_1(t)+2\geq\de_3\quad\forall t\geq t_0.
\end{gather}
\end{asu}

\begin{rem}
We will discuss in the later subsection how to select functions satisfying Assumption \ref{asu-main}. For further analysis, we make the following remarks.

(1) Under Assumption \ref{asu-main}, the lower and upper bounds of $\be_0,\be_1,\be_2$ exist and they are denoted as
\begin{gather}
\label{cond-main-4}0<c_0:=\inf\be_0\leq\sup\be_0=:\al_0<\infty,\\
\label{cond-main-5}0<c_1:=\inf\be_1\leq\sup\be_1=:\al_1<\infty,\\
\label{cond-main-6}0<c_2:=\inf\be_2\leq\sup\be_2=:\al_2<\infty.
\end{gather}
Moreover, we can find positive constants $\de_4,\de_5,\de_6$ satisfying the following conditions
\begin{gather}
\label{cond7}A_2(t)\leq\de_4\quad\forall t\geq t_0,\\
\label{cond8}B_1(t)\geq\de_5\quad\forall t\geq t_0,\\
\label{cond9}C_0(t)\geq\de_6\quad\forall t\geq t_0.    
\end{gather}

(2) Under Assumption \ref{asu-main}, the first-order derivatives of $\be_0,\be_1,\be_2$ are bounded.
\begin{gather}\label{202303311326}
\begin{cases}
\be_2'(t_0)\leq\be_2'(t)\leq 0,\\
\be_1'(t_0)\leq\be_1'(t)\leq 0,\\
0\leq\be_0'(t)\leq\be_0'(t_0).
\end{cases}    
\end{gather}

(3) Under Assumption \ref{asu-main}, the following inequalities hold
\begin{align}
\label{cond-main-1}&\be_2''\geq 0\geq\be_1',\\
\nonumber &A_2''=-D''\frac{\be_1}{\be_0}+2D'\frac{\be_1\be_0'-\be_1'\be_0}{\be_0^2}+(2\om_u-D)\left(\frac{\be_1''}{\be_0}-2\frac{\be_1'\be_0'}{\be_0^2}-\frac{\be_1\be_0''}{\be_0^2}+\frac{2\be_1}{\be_0^3}(\be_0')^2\right)\geq 0,\\
\nonumber&-A_1'=(2\om_u-D)\left(-\frac{\be_2'\be_1}{\be_0}-\frac{\be_2\be_1'}{\be_0}+\frac{\be_2\be_1\be_0'}{\be_0^2}\right)+\frac{1}{D}\left(\be_0'\ld_1\ld_2+\be_0\ld_1'\ld_2+\be_0\ld_1\ld_2'\right)\\
&\quad\quad\nonumber+D'\left(\frac{\be_1\be_2}{\be_0}-\frac{\be_0}{D^2}\ld_1\ld_2\right)\geq 0,\\
\label{cond-main-2}&A_2''-A_1'+A_0\geq A_0,\\
\label{cond-main-3}&B_0-B_1'=B_0+(2\om_u-D)\frac{\be_2\be_0'-\be_2'\be_0}{\be_0^2}+D'\frac{\be_2}{\be_0}\geq B_0,\\
\label{cond-main-7}&A_1-A_2'+2=A_1+(2\om_u-D)\left(\frac{\be_1\be_0'-\be_1'\be_0}{\be_0^2}\right)+D'\frac{\be_1}{\be_0}+2\geq A_1+2,\\
\label{cond-main-8}&\be_1-\be_2'\geq\be_1\geq c_1.
\end{align}
\end{rem}

We are now in a position to establish the weak convergence of the trajectories generated by the dynamical system \eqref{ode3}.

\begin{thm}\label{main-res1}
Consider the dynamical system \eqref{ode3}, where the operator $U:\calh\to\calh$ is $\om_u$-quasi-cocoercive and is $L_u$-Lipschitz continuous. Suppose Assumption \ref{asu-main} and the following condition
\begin{gather}\label{cond-main-9}
 \frac{1}{c_1}<\de_5-\frac{\de_4^2}{\de_3}.
\end{gather}
Then the following results hold.
\begin{enumerate}[label=(\roman*)]
    \item $x,x',x''$ are bounded.
    \item $x',x'',x''',U(x)\in L^2$.
    \item $\lim\limits_{t\to\infty}x'(t)=\lim\limits_{t\to\infty}x''(t)=\lim\limits_{t\to\infty}U(x(t))=0.$
    \item {\color{black} If in addition, the operator $U$ is demiclosed at $0$}, then the trajectory $x(\cdot)$ generated by the dynamical system \eqref{ode3} converges weakly to an element in $Z(U).$
\end{enumerate}
\end{thm}
\begin{proof}
We make use of the notations already formulated in \eqref{202303291322}, \eqref{cond4}-\eqref{cond6} and \eqref{cond7}-\eqref{cond9}. The convergence of equation \eqref{ode3} is studied by using the functions defined in \eqref{L-func}. By \eqref{cond-main-9}, we can take $\varepsilon>0$ subject to the following
\begin{gather}\label{202330032124}
 \frac{1}{c_1}<\frac{1}{\varepsilon}<\de_5-\frac{\de_4^2}{\de_3}.   
\end{gather}
Since
\begin{align}
    \nonumber y'(t)&=2\inner{x'(t)}{x(t)-x_*},\\
    \label{202303291351}y''(t)&=2\inner{x''(t)}{x(t)-x_*}+2z_1(t),\\
    \nonumber y'''(t)&=2\inner{x'''(t)}{x(t)-x_*}+3z_1'(t),
\end{align}
by \eqref{ode3}, we have
\begin{align*}
&\quad\,\,y'''(t)+\be_2(t)y''(t)+\be_1(t)y'(t)\\
&=-2\be_0(t)\inner{U(\phi_x(t))}{x(t)-x_*}
+3z_1'(t)+2\be_2(t)z_1(t)\\
&=-2\be_0(t)\inner{U(\phi_x(t))}{\phi_x(t)-x_*}+2\be_0(t)\inner{U(\phi_x(t))}{\ld_1(t)x'(t)+\ld_2(t)x''(t)}\\
&\quad\quad\quad+3z_1'(t)+2\be_2(t)z_1(t).    
\end{align*}
Since the operator $U$ is $\om_u$-quasi-cocoercive and by the Cauchy-Schwarz inequality, we can estimate
\begin{align}
\nonumber &\quad\,\,y'''(t)+\be_2(t)y''(t)+\be_1(t)y'(t)\\
\nonumber&\leq-(2\om_u-D(t))\be_0(t)\|U(\phi_x(t))\|^2+\frac{\be_0(t)}{D(t)}\|\ld_1(t)x'(t)+\ld_2(t)x''(t)\|^2\\
\nonumber&\quad\quad\quad+3z_1'(t)+2\be_2(t)z_1(t)\\
\label{202303291323}&=-\frac{2\om_u-D(t)}{\be_0(t)}\|x'''(t)+\be_2(t)x''(t)+\be_1(t)x'(t)\|^2\\
\nonumber&\quad\quad\quad+\frac{\be_0(t)}{D(t)}\|\ld_1(t)x'(t)+\ld_2(t)x''(t)\|^2+3z_1'(t)+2\be_2(t)z_1(t).    
\end{align}
A direct computation gives
\begin{align}
\label{202303291324}&\|x'''+\be_2x''+\be_1x'\|^2
=\be_1z_1''+\be_1\be_2z_1'+\be_1^2z_1+\be_2z_2'+(\be_2^2-2\be_1)z_2+z_3,\\
\label{202303291324a}&\|\ld_1x'+\ld_2x''\|^2=\ld_1^2z_1+\ld_2^2z_2+\ld_1\ld_2z_1'.    
\end{align}
By substituting \eqref{202303291324}-\eqref{202303291324a} back into \eqref{202303291323} and then using the notations in \eqref{202303291322}, we get
\begin{gather}
\label{202303310943}0\geq y'''+\be_2y''+\be_1y'+A_2z_1''+A_1z_1'+A_0z_1+B_1z_2'+B_0z_2+C_0z_3.
\end{gather}
After integrating the line above over $[t_0,t]$, there exists a constant $M_1=M_1(t_0)$ such that
\begin{align}
\nonumber M_1&\geq y''(t)+\be_2(t)y'(t)+[\be_1(t)-\be_2'(t)]y(t)+A_2(t)z_1'(t)\\
\nonumber&\quad\quad\quad+[A_1(t)-A_2'(t)]z_1(t)+B_1(t)z_2(t)\\
\nonumber&\quad\quad\quad+\int\limits_{t_0}^t[\be_2''(s)-\be_1'(s)]y(s)\,ds
+\int\limits_{t_0}^t[A_2''(s)-A_1'(s)+A_0(s)]z_1(s)\,ds\\
\label{202303310831}&\quad\quad\quad+\int\limits_{t_0}^t[B_0(s)-B_1'(s)]z_2(s)\,ds+\int\limits_{t_0}^tC_0(s)z_3(s)\,ds.    
\end{align}
Using \eqref{202303291351}, \eqref{cond-main-1}, \eqref{cond-main-2} and \eqref{cond-main-3}, we get
\begin{align}
\nonumber M_1&\geq 2\inner{x''}{x-x_*}+\be_2y'+(\be_1-\be_2')y
+A_2z_1'+(A_1-A_2'+2)z_1+B_1z_2\\
\label{202303302101}&\geq\be_2y'+(\be_1-\be_2'-\varepsilon)y+A_2z_1'+(A_1-A_2'+2)z_1+\left(B_1-\frac{1}{\varepsilon}\right)z_2,    
\end{align}
where we estimate $2\inner{x''}{x-x_*}$ by using the Cauchy-Schwarz inequality. Setting
\begin{gather*}
H:=A_2z_1'+(A_1-A_2'+2)z_1+\left(B_1-\frac{1}{\varepsilon}\right)z_2.    
\end{gather*}
Then
\begin{align}
\nonumber H&=2A_2\inner{x''}{x'}+(A_1-A_2'+2)\|x'\|^2+\left(B_1-\frac{1}{\varepsilon}\right)\|x''\|^2\\
\nonumber&\geq 2A_2\inner{x''}{x'}+\de_3\|x'\|^2+\left(\de_5-\frac{1}{\varepsilon}\right)\|x''\|^2\quad\text{(by \eqref{cond-main-7}, \eqref{cond6}, \eqref{cond8})}\\
\nonumber&\geq -2A_2\|x''\|\cdot\|x'\|+\de_3\|x'\|^2+\left(\de_5-\frac{1}{\varepsilon}\right)\|x''\|^2\\
\label{202303302059}&\geq -2\de_4\|x''\|\cdot\|x'\|+\de_3\|x'\|^2+\left(\de_5-\frac{1}{\varepsilon}\right)\|x''\|^2\quad\text{(by \eqref{cond7})}.    
\end{align}
Note that
\begin{align}
\nonumber&\quad\,\,-2\de_4\|x''\|\cdot\|x'\|+\de_3\|x'\|^2+\left(\de_5-\frac{1}{\varepsilon}\right)\|x''\|^2\\
\label{202303302049}&=\left(\sqrt{\de_3}\|x'\|-\frac{\de_4}{\sqrt{\de_3}}\|x''\|\right)^2
+\left(\de_5-\frac{\de_4^2}{\de_3}-\frac{1}{\varepsilon}\right)\|x''\|^2\\
\label{202303302050}&=\left(\|x''\|\sqrt{\de_5-\frac{1}{\varepsilon}}-\|x'\|\frac{\de_4}{\sqrt{\de_5-\frac{1}{\varepsilon}}}\right)^2
+\left(\de_3-\frac{\de_4^2}{\de_5-\frac{1}{\varepsilon}}\right)\|x'\|^2.    
\end{align}
Taking the sum of \eqref{202303302049} and \eqref{202303302050}, we can find
\begin{gather}\label{202303302100}
-2\de_4\|x''\|\cdot\|x'\|+\de_3\|x'\|^2+\left(\de_5-\frac{1}{\varepsilon}\right)\|x''\|^2
\geq\theta(\|x''\|^2+\|x'\|^2),
\end{gather}
where
\begin{gather*}
\theta=\frac{1}{2}\min\left\{\de_5-\frac{\de_4^2}{\de_3}-\frac{1}{\varepsilon},\,\,\de_3-\frac{\de_4^2}{\de_5-\frac{1}{\varepsilon}}\right\}.    
\end{gather*}
Inequalities \eqref{202303302059} and \eqref{202303302100} give
\begin{gather}\label{202303302107}
H\geq\theta(\|x''\|^2+\|x'\|^2).   
\end{gather}
From \eqref{202303302101} and \eqref{202303302107}, we get
\begin{gather*}
M_1\geq\be_2y'+(\be_1-\be_2'-\varepsilon)y+\theta(\|x''\|^2+\|x'\|^2),    
\end{gather*}
which implies, after dividing both sides by $\be_2$, that
\begin{gather*}
\frac{M_1}{\be_2}\geq y'+\frac{1}{\be_2}(\be_1-\be_2'-\varepsilon)y+\frac{1}{\be_2}\theta(\|x''\|^2+\|x'\|^2).    
\end{gather*}
Using \eqref{cond-main-6} and \eqref{cond-main-8}, the line above can be estimated as follows
\begin{align}
\nonumber\frac{M_1}{c_2}\geq\frac{M_1}{\be_2}&\geq y'+\frac{1}{\be_2}(\be_1-\be_2'-\varepsilon)y+\frac{1}{\be_2}\theta(\|x''\|^2+\|x'\|^2)\\
\label{202303310810}&\geq y'+\frac{c_1-\varepsilon}{\al_2}y+\frac{\theta}{\al_2}(\|x''\|^2+\|x'\|^2).    
\end{align}
(i) Setting $\eta:=\frac{c_1-\varepsilon}{\al_2}$. Then by \eqref{202330032124}, we have $\eta>0$. By multiplying the inequality above with $e^{\eta s}$ and then integrating over $[t_0,t]$, we obtain
\begin{gather*}
    y(t)\leq y(t_0)e^{\eta(t_0-t)}+\frac{M_1}{\eta c_2}(1-e^{\eta(t_0-t)})\leq y(t_0)e^{\eta t_0}+\frac{M_1}{\eta c_2},
\end{gather*}
which means that $x$ is bounded.

Inequality \eqref{202303310810} leads to the following
\begin{gather*}
\frac{M_1}{c_2}\geq y'+\frac{\theta}{\al_2}\|x'\|^2=2\inner{x'}{x-x_*}+\frac{\theta}{\al_2}\|x'\|^2\geq\left(\frac{\theta}{\al_2}-\varepsilon_0\right)\|x'\|^2  -\frac{1}{\varepsilon_0}\|x-x_*\|^2,
\end{gather*}
where $\varepsilon_0:=\frac{\theta}{2\al_2}$ and the expression $2\inner{x'}{x-x_*}$ is estimated by the Cauchy-Schwarz inequality. The inequality above shows $x'$ is bounded. Since $x,x'$ are bounded, by \eqref{202303310810} $x''$ is bounded too.

(ii) Inequality \eqref{202303310831} gives
\begin{align*}
M_1&\geq y''(t)+\be_2(t)y'(t)+[\be_1(t)-\be_2'(t)]y(t)+H-2z_1+\frac{z_2}{\varepsilon}\\
&\quad\quad\quad+\int\limits_{t_0}^t[A_2''(s)-A_1'(s)+A_0(s)]z_1(s)\,ds\\
&\quad\quad\quad+\int\limits_{t_0}^t[B_0(s)-B_1'(s)]z_2(s)\,ds+\int\limits_{t_0}^tC_0(s)z_3(s)\,ds,    
\end{align*}
which implies, by \eqref{cond-main-8}, \eqref{202303302107}, \eqref{cond-main-2}, \eqref{cond-main-3}, \eqref{cond-main-4}, that
\begin{align*}
M_1&\geq y''(t)+\be_2(t)y'(t)+c_1y(t)-2z_1\\
&\quad\quad\quad+\de_1\int\limits_{t_0}^tz_1(s)\,ds+\de_2\int\limits_{t_0}^tz_2(s)\,ds+\de_6\int\limits_{t_0}^tz_3(s)\,ds.    
\end{align*}
Thus, $x',x'',x'''\in L^2$.

It follows from \eqref{ode3} that
\begin{align*}
\|U(\phi_x(t))\|&=\frac{1}{\be_0(t)}\|x'''(t)+\be_2(t)x''(t)+\be_1(t)x'(t)\|\\
&\leq\frac{1}{c_0}\|x'''(t)+\be_2(t)x''(t)+\be_1(t)x'(t)\|\\
&\leq\frac{1}{c_0}(\|x'''(t)\|+\al_2\|x''(t)\|+\al_1\|x'(t)\|).    
\end{align*}
Hence, $U(\phi_x(t))\in L^2,$ too. 

Next is to prove $U(x(t))\in L^2$. Indeed, by \eqref{202404280817} and \eqref{202404280818}, we have
\begin{gather*}
(2\om_u)^2>\frac{\be_0(t)^2\ld_1(t)\ld_2(t)}{\be_1(t)\be_2(t)}\geq\frac{\be_0(t_0)^2}{\be_1(t_0)\be_2(t_0)}\ld_1(t)\ld_2(t),
\end{gather*}
which means, by \eqref{202404280820}-\eqref{202404280821}, that $\ld_1,\ld_2$ are bounded above. Since the operator $U:\calh\to\calh$ is $L_u$-Lipschitz continuous, we have
\begin{align*}
\|U(\phi_x(t))-U(x(t))\|&\leq L_u\|\ld_1(t)x'(t)+\ld_2(t)x''(t)\|\\
&\leq L_u\left(\|x'(t)\|\sup\limits_{t\geq t_0}\ld_1(t)+\|x''(t)\|\sup\limits_{t\geq t_0}\ld_2(t)\right),       
\end{align*}
which gives $U(x(t))\in L^2.$

(iii) We make use of Lemma \ref{lem-main-1}. Indeed, we have
\begin{gather*}
\frac{d}{dt}\left(\frac{1}{2}\|x'(t)\|^2\right)=\inner{x''(t)}{x'(t)}\leq\frac{1}{2}(\|x''(t)\|^2+\|x'(t)\|^2)\in L^1,\\
\frac{d}{dt}\left(\frac{1}{2}\|x''(t)\|^2\right)=\inner{x'''(t)}{x''(t)}\leq\frac{1}{2}(\|x'''(t)\|^2+\|x''(t)\|^2)\in L^1,
\end{gather*}
and
\begin{align*}
\frac{d}{dt}\left(\frac{1}{2}\|U(x(t))\|^2\right)&=\inner{\frac{d}{dt}U(x(t))}{U(x(t))}\\
&\leq\frac{1}{2}\left(\left\|\frac{d}{dt}U(x(t))\right\|^2+\|U(x(t))\|^2\right)\\
&\leq\frac{1}{2}\left(L_u^2\|x'(t)\|^2+\|U(x(t))\|^2\right)\in L^1.    
\end{align*}

(iv) We make use of Lemma \ref{lem-op}. Indeed, inequality \eqref{202303310943} can be expressed as follows
\begin{align*}
0&\geq  y'''+\frac{d}{dt}(\be_2y')+\frac{d}{dt}[(\be_1-\be_2')y]+\frac{d}{dt}(A_2z_1')+\frac{d}{dt}[(A_1-A_2')z_1]+\frac{d}{dt}(B_1z_2)\\
&\quad\quad\quad+(\be_2''-\be_1')y+(A_2''-A_1'+A_0)z_1+(B_0-B_1')z_2\\
&\geq y'''+\frac{d}{dt}(\be_2y')+\frac{d}{dt}[(\be_1-\be_2')y]+\frac{d}{dt}(A_2z_1')+\frac{d}{dt}[(A_1-A_2')z_1]+\frac{d}{dt}(B_1z_2),   
\end{align*}
which means that the function $h:=y''+\be_2y'+(\be_1-\be_2')y+A_2z_1'+(A_1-A_2')z_1+B_1z_2$ is monotonically decreasing. By \eqref{202303311326} and item (i), the function $h$ is bounded. Consequently, the limit $\lim\limits_{t\to\infty}h(t)$ exists. Note that items (i)-(iii) and \eqref{202404280817}, \eqref{202404280819}, \eqref{cond6}, \eqref{202303311326}, \eqref{cond-main-7} reveal that 
\begin{gather*}
\lim\limits_{t\to\infty}[y''(t)+\be_2(t)y'(t)+A_2(t)z_1'(t)+(A_1(t)-A_2'(t))z_1(t)+B_1(t)z_2(t)]=0.   
\end{gather*}
Thus, the limit $\lim\limits_{t\to\infty}[\be_1(t)-\be_2'(t)]y(t)$ exists. By \eqref{cond-main-1} and \eqref{cond-main-8}, the limit $\lim\limits_{t\to\infty}[\be_1(t)-\be_2'(t)]$ exists and it is positive. We infer that $\lim\limits_{t\to\infty}y(t)$ exists for every $x_*\in Z(U)$.

Next, we prove Lemma \ref{lem-op}(ii). To that aim, let $\widehat{x}$ be a weak sequential cluster point of $x;$ meaning that there exists a sequence $s_n\to\infty$ such that $x(s_n)$ converges weakly to $\widehat{x}$. 
Since $U$ is demiclosed at zero
and $\lim\limits_{n\to\infty}U(x(s_n))=0$ we deduce $U(\widehat{x})=0.$ The proof is complete.
\end{proof}

\subsection{Parameters choices}
In the subsection, we give examples illustrating Theorem \ref{main-res1}. Let
\begin{gather}\label{202304011400}
\begin{cases}
\be_j(t)=p_je^{-r_jt}+q_j,\quad\be_0(t)=\frac{q_0}{p_0e^{-r_0t}+1},\\
\ld_j(t)=\ta_j(1-e^{-m_j t}).
\end{cases}
\end{gather}

\begin{thm}\label{thm-pe+q}
Consider equation \eqref{ode3}, where the operator $U:\calh\to\calh$ is $\om_u$-quasi-cocoercive, $L_u$-Lipschitz continuous and the coefficients are of forms in \eqref{202304011400}. Here
\begin{gather}
q_j>0,\ta_j>0,p_j\geq 0,r_j\geq 0,m_j\geq 0.    
\end{gather}
Then the trajectory $x$ generated by \eqref{ode3} converges weakly to an element in $Z(U)$ provided that
\begin{align}
\label{202304020908}&q_2>\max\left\{\frac{p_1+q_1}{\sqrt{q_1}};\,\,\sqrt{2(p_1+q_1)}\right\},\\
\label{202304020909}&p_0<\frac{q_2\sqrt{q_1}}{p_1+q_1}-1,\\
\label{202304020910}&\frac{\om_u}{q_0}>\max\left\{\frac{2(p_2+q_2)}{q_1^2};\,\,\frac{1}{q_1q_2};\,\,\frac{q_2+\frac{(p_1+q_1)(p_0+1)}{\sqrt{q_1}}}{q_1q_2^2-(p_1+q_1)^2(p_0+1)^2}\right\},\\
\label{202405200907}&0\leq\ta_1<\sqrt{\frac{\om_u}{q_0}\left(\frac{\om_u q_1^2}{q_0}-2(p_2+q_2)\right)},\\
\label{202405200908}&0\leq\ta_1\ta_2<\frac{\om_u}{q_0}\cdot\min\left\{\frac{\om_u q_1q_2}{q_0}-1;\,\,\ta_1\sqrt{q_2^2-2(p_1+q_1)}\right\},\\
\label{202405252203}&0\leq\ta_1\ta_2<\frac{\om_uq_1}{\om_uq_1q_2-q_0}\left((q_1q_2^2-(p_1+q_1)^2(p_0+1)^2)\left(\frac{\om_u}{q_0}\right)^2-2q_2\frac{\om_u}{q_0}+\frac{1}{q_1}\right).    
\end{align}
\end{thm}
\begin{proof}
A straightforward computation shows
\begin{gather*}
\begin{cases}
c_0=\frac{q_0}{p_0+1},\\
c_1=q_1,\\
c_2=q_2,
\end{cases}
\quad\quad
\begin{cases}
\al_0=q_0,\\
\al_1=p_1+q_1,\\
\al_2=p_2+q_2.
\end{cases}
\end{gather*}
Let $D(t)=\om_u$. Note that condition \eqref{202404280818} follows directly from \eqref{202405200908}. It follows from \eqref{202304020908} and \eqref{202405200908} that
\begin{gather*}
0<\de_2:=\frac{\om_u}{q_0}[q_2^2-2(p_1+q_1)]-\frac{q_0\ta_2^2}{\om_u}\leq B_0,   
\end{gather*}
which gives \eqref{cond5}. By \eqref{202405200907}, we get
\begin{gather*}
0<\de_1:=\frac{\om_u}{q_0}q_1^2-2(p_2+q_2)-\frac{q_0\ta_1^2}{\om_u}\leq A_0,   
\end{gather*}
which gives \eqref{cond4}. Again using \eqref{202405200908}, we observe $\ta_1\ta_2<\frac{\om_u}{q_0}(\frac{\om_uq_1q_2}{q_0}-1)$. Consequently,
\begin{gather*}
0<\de_3:=\om_u\frac{q_1q_2}{q_0}-1-\frac{q_0\ta_1\ta_2}{\om_u}\leq A_1+2,
\end{gather*}
which gives \eqref{cond6}. We have
\begin{align*}
&A_2=\frac{\om_u\be_1}{\be_0}\leq\frac{\om_u\al_1}{c_0}=\frac{\om_u}{q_0}(p_1+q_1)(p_0+1)=:\de_4,\\
&B_1=\frac{\om_u\be_2}{\be_0}\geq\frac{\om_uq_2}{q_0}=:\de_5.    
\end{align*}
Consider
\begin{gather*}
\left[q_1q_2^2-(p_1+q_1)^2(p_0+1)^2\right]\left(\frac{\om_u}{q_0}\right)^2-2q_2\frac{\om_u}{q_0}+\frac{1}{q_1},
\end{gather*}
which is a quadratic form of the variable $\frac{\om_u}{q_0}$ with the leading coefficient
\begin{gather*}
q_1q_2^2-(p_1+q_1)^2(p_0+1)^2>0\quad\text{(by \eqref{202304020909})}.    
\end{gather*}
Note that condition \eqref{202304020910} gives
\begin{gather*}
\frac{\om_u}{q_0}>\frac{q_2+\frac{(p_1+q_1)(p_0+1)}{\sqrt{q_1}}}{q_1q_2^2-(p_1+q_1)^2(p_0+1)^2}.
\end{gather*}
Consequently, the quadratic form is always positive. Note that condition \eqref{cond-main-9} is equivalent to \eqref{202405252203}.
\end{proof}

\begin{rem}
    It is not difficult to choose a set of parameters satisfying the system of inequalities \eqref{202304020908}--\eqref{202405252203}. Indeed, given $p_1, q_1$ we first choose $q_2$ satisfying \eqref{202304020908}, then select $p_0$ so that \eqref{202304020909} holds and so on until \eqref{202405252203}.
    
\end{rem}

In the case when $p_j=0$ for $j\in\{0,1,2\},$ \eqref{ode3} is simplified to the third-order ODE with constant coefficients and Theorem \ref{thm-pe+q} becomes the following result.
\begin{cor}\label{cor-cons-coe}
Consider equation \eqref{ode3}, where the operator $U:\calh\to\calh$ is $\om_u$-quasi-cocoercive, $L_u$-Lipschitz continuous and $\be_j\equiv q_j$, $\ld_j\equiv 0$. Then the trajectory $x$ generated by \eqref{ode3} converges weakly to an element in $Z(U)$ provided that
\begin{gather} \label{con}
q_2>\sqrt{2q_1},\\ \label{con2}
q_0<\om_u\frac{q_1^2}{2q_2}.
\end{gather} 
\end{cor}

\begin{exa}\label{examMonotone}
    Let us consider an interesting  example to show the effectiveness of considering third-order dynamical systems over the first- and second-order dynamical systems. We take $U(x):= A x$ where $A$ is a random positive semi-definite matrix in $\mathbb{R}^{n \times 
 n}$. We make sure that $A$ is singular, so that it is not strongly monotone and Problem \eqref{pb-zer} may have multiple  solutions. Note that in this case solving Problem \eqref{pb-zer} is equivalent to solving the quadratic convex optimization problem $\min_{x\in \mathbb{R}^n} f(x):= \frac{1}{2}\langle x, Ax \rangle$. Hence, $U$ is Lipschitz continuous with $L_u=\lambda_{\max}$, the maximum eigenvalue of $A$, and cocoercive with $\om_u = 1/L_u$. We consider three dynamical systems with initial conditions
$$
x(t_0) = \text{normalized random vector }, \quad x'(t_0)= (0,0...,0), \quad x''(t_0)=(0,0...,0),
$$
as follows
$$
x'(t)+U(x(t))=0;
$$
$$
x''(t) +\frac{3}{t} x'(t)+U(x(t))=0;
$$
and 
$$
x'''(t)+ \beta_2 x''(t)+ \beta_1 x'(t)+  U(x(t))=0,
$$
where
$$
\beta_1 = 2.1 L_u^{2/3}, \quad \beta_2 = 1.5 \sqrt{\beta_1} 
$$
satisfying conditions \eqref{con}-\eqref{con2}. 
The error $\|A(x(t))\|$ plotted in Figure \ref{figExammono} (with $50$ random samples of $A$) clearly shows that the trajectories generated by the third-order dynamical system converge faster than those of the first- and second-order dynamical systems.

\begin{figure}[h]
    \centering
\includegraphics[width=0.48\textwidth]{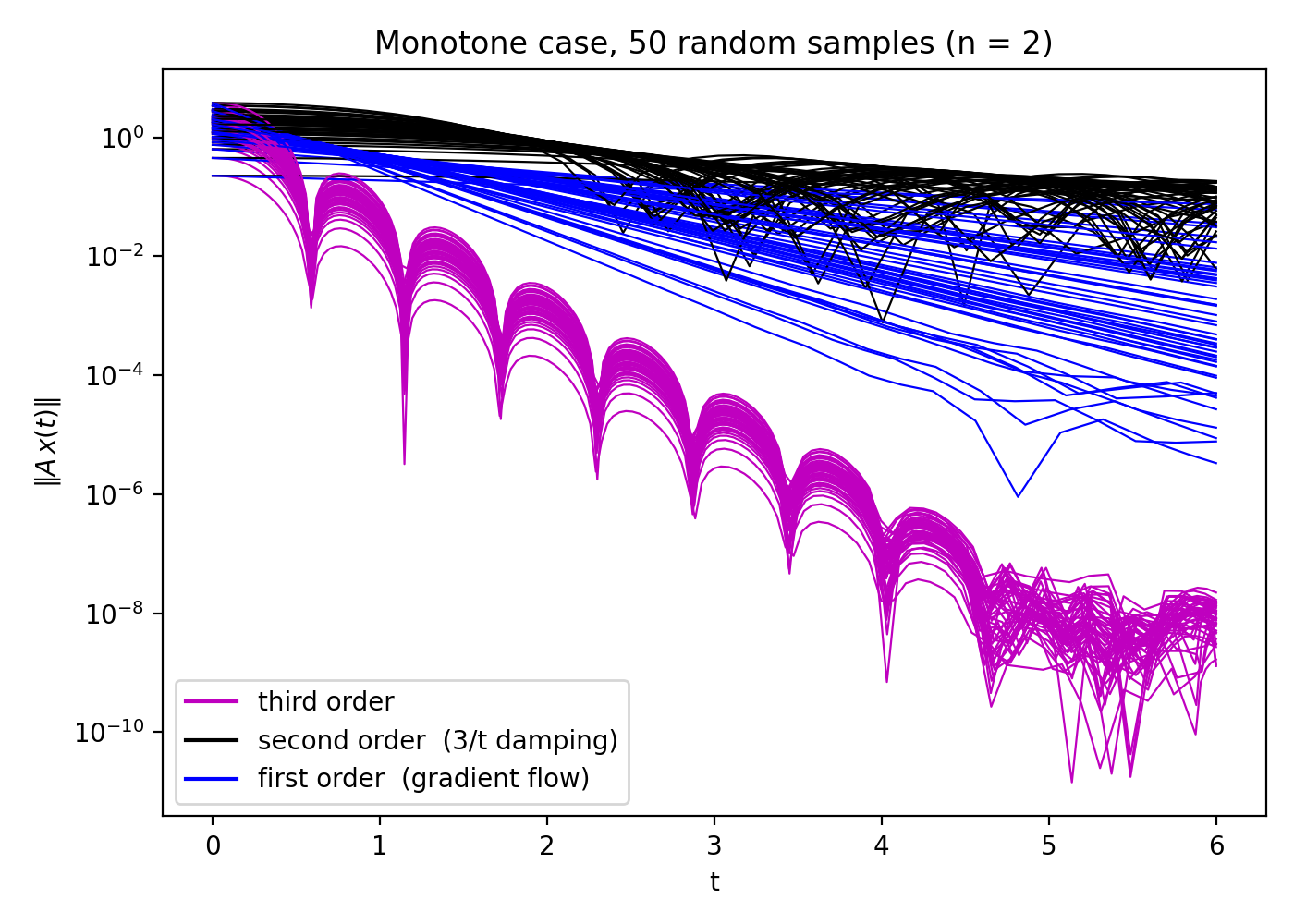} 
\includegraphics[width=0.48\textwidth]{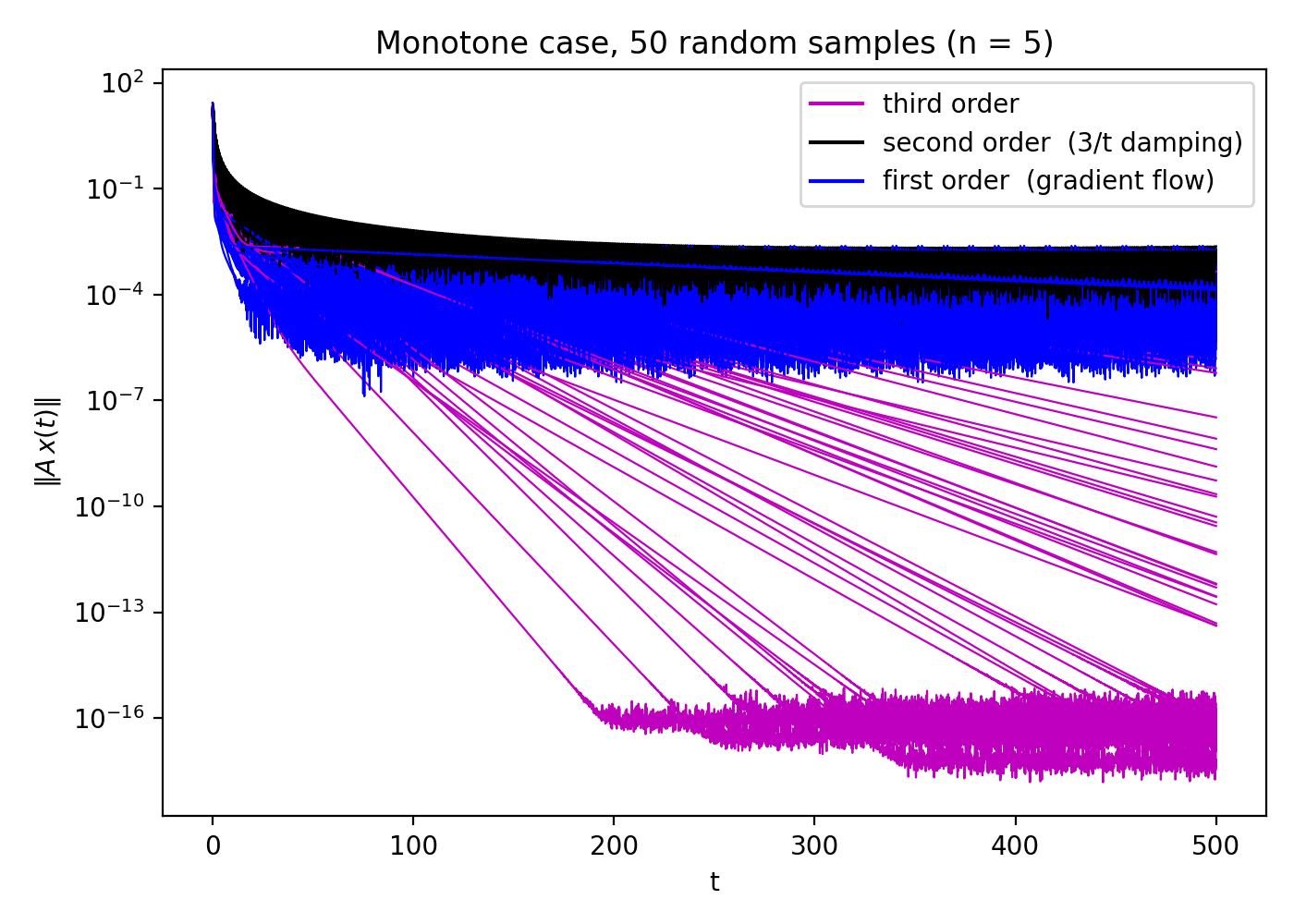} 
    \caption{The error $\|A(x(t))\|$  in Example \ref{examMonotone} with $n=2$ (left) and  $n=5$ (right) for 50 random samples.}
    \label{figExammono}
\end{figure}
\end{exa}


\section{Exponential convergence}\label{ex}
In this section, we estimate the exponential convergence of the trajectory generated by the dynamical system \eqref{ode3}.
\begin{asu}\label{asu-202405161507a}
Assume that
\begin{enumerate}[label=(\roman*)]
\item The operator $U:\calh\to\calh$ is $\rho$-strongly monotone with respect to $Z(U)$, i.e.
$$
\inner{U(x)}{x-x_*} \geq \rho\|x-x_*\|^2\quad \forall x_*\in Z(U), \forall x\in\calh.
$$
\item The operator $U:\calh\to\calh$ is $L_u$-Lipschitz continuous.
\item The coefficients $\be_j,\ld_j$ of the dynamical system \eqref{ode3} are constants.
\end{enumerate}
\end{asu}

\begin{thm}\label{thm20240523}
Consider the dynamical system \eqref{ode3}, under Assumption \ref{asu-202405161507a}. Denote
\begin{gather}\label{202405232146}
\kappa:=\frac{\rho}{L_u^2},
\end{gather}
where $\rho,L_u$ are the constants stated in Assumption \ref{asu-202405161507a}. Then the trajectory $x(\cdot)$ generated by the dynamical system \eqref{ode3} converges strongly to the unique solution $x_*$ at the rate 
\begin{align*}
\|x(t)-x_*\|^2=\mathcal{O}(e^{-3t}),
\end{align*}
provided conditions \eqref{202405172036}-\eqref{202405231732} hold.
\begin{align}
\label{202405172036}&\be_2>\frac{32}{\rho\kappa},\\
\label{202405172037}&(\be_2-5)\frac{16}{\rho\kappa}<\be_1<\frac{(\be_2-4)\be_2}{2},\\
\label{202405172038}&\frac{4}{\rho}(-4\be_2+\be_1+16)<\be_0<\frac{\kappa\be_1}{2}\cdot\frac{-4\be_2+\be_1+16}{2\be_2-12},\\
\label{202405231728}&0\leq\ld_1<\frac{-64+\rho\be_0-4\be_1+16\be_2}{4\rho\be_0},\\
\label{202405231729}&\be_0\left(\frac{2}{\kappa}-\rho\right)\ld_1^2+2\be_0\rho\ld_2<\frac{\kappa\be_1}{2\be_0}(16+\be_1-4\be_2)-(2\be_2-12),\\
\label{202405231730}&0\leq\ld_2<\frac{37+\be_1-7\be_2}{7\rho\be_0},\\
\label{202405231731}&\be_0\left(\frac{2}{\kappa}-\rho\right)\ld_1\ld_2<\frac{\kappa\be_1}{2\be_0}(\be_2-7)-3,\\
\label{202405231732}&\be_0\left(\frac{2}{\kappa}-\rho\right)\ld_2^2<\frac{\kappa}{2\be_0}(\be_2^2-4\be_2-2\be_1).    
\end{align}   
\end{thm}
\begin{proof}
First, we verify the feasibility of conditions \eqref{202405172037}-\eqref{202405231732}. Note that
\begin{gather*}
\rho\kappa<1\Longrightarrow    \frac{2}{\kappa}-\rho>0.
\end{gather*}
It results from \eqref{202405172037} that $0<\be_2-5<\be_2-4$ and $\frac{16}{\rho\kappa}<\frac{\be_2}{2}$. Then multiplying these inequalities yields 
\begin{align*}
(\be_2-5)\frac{16}{\rho\kappa}<\frac{(\be_2-4)\be_2}{2},
\end{align*}
which proves the validity of \eqref{202405172037}. Once \eqref{202405172037}-\eqref{202405172036} are met, we have
\begin{align*}
\be_1>(\be_2-5)\frac{16}{\rho\kappa}>(\be_2-5)7>4\be_2-16   
\end{align*}
and so $-4\be_2+\be_1+16>0.$ The first inequality of \eqref{202405172037} gives
\begin{align*}
\be_1>(\be_2-5)\frac{16}{\rho\kappa}>(\be_2-6)\frac{16}{\rho\kappa},
\end{align*}
which gives
\begin{align*}
\frac{4}{\rho}<\frac{\kappa\be_1}{2}\cdot\frac{1}{2\be_2-12}.
\end{align*}
The fulfillment of \eqref{202405172038} follows after multiplying both sides of the inequality above by $-4\be_2+\be_1+16>0.$ The positivity of the expression $-64+\rho\be_0-4\be_1+16\be_2$ stated in \eqref{202405231728} follows directly from \eqref{202405172038}. The second inequality in \eqref{202405172038} is equivalent to saying
\begin{align*}
\frac{\kappa\be_1}{2\be_0}(16+\be_1-4\be_2)-(2\be_2-12)>0
\end{align*}
and the second inequality in \eqref{202405172037} to
\begin{align*}
\be_2^2-4\be_2-2\be_1>0.
\end{align*}
By \eqref{202405172037}, we observe
\begin{align*}
\be_1>(\be_2-5)7>7\be_2-37,  
\end{align*}
which leads to the positivity of $37+\be_1-7\be_2$ stated in \eqref{202405231730}. The second inequality of \eqref{202405172037} gives
\begin{align*}
\be_1<\frac{\be_2(\be_2-4)}{2}<\frac{1}{3}(2\be_2^2-14\be_2+36)
\end{align*}
and then
\begin{align*}
\frac{\kappa\be_1}{6}(\be_2-7)>\frac{\kappa\be_1}{2}\cdot\frac{-4\be_2+\be_1+16}{2\be_2-12}>\be_0,
\end{align*}
where we make use of \eqref{202405172038}. Consequently,
\begin{align*}
\frac{\kappa\be_1}{2\be_0}(\be_2-7)-3>0.
\end{align*}
Since the right-hand sides of \eqref{202405231728}-\eqref{202405231732} are all strictly positive, sufficiently small parameters $\lambda_1, \lambda_2$ can be chosen to satisfy \eqref{202405231728}-\eqref{202405231732}.

From Assumption \ref{asu-202405161507a}, it is clear that $Z(U)$ is singleton. Let $x_*\in Z(U)$. Then 
for any $\overline{v}\in\calh$ we have
\begin{gather}\label{202405210908}
\langle U(\overline{v}), \overline{v} -x_* \rangle \ge \rho \|\overline{v}-x_*\|^2.
\end{gather}
By the Cauchy-Schwarz inequality, we deduce
$$
\|U(\overline{v})\|\cdot\|\overline{v}-x_*\|
\ge 
\langle U(\overline{v}), \overline{v} -x_* \rangle \ge \rho \|\overline{v}-x_*\|^2,
$$
which implies 
$$
\|U(\overline{v})\| \ge \rho \|\overline{v}-x_*\|.
$$
Moreover, by the Lipschitz continuity 
$$
\|U(\overline{v})\|^2 \le L_u^2 \|\overline{v}-x_*\|^2 
\le \frac{L_u^2}{\rho} \langle U(\overline{v}), \overline{v} -x_* \rangle,
$$
or equivalently 
\begin{gather}\label{202405210909}
\langle U(\overline{v}), \overline{v} -x_* \rangle\ge \frac{\rho}{L_u^2}\|U(\overline{v})\|^2=\kappa\|U(\overline{v})\|^2.
\end{gather}
Note that
\begin{gather}\label{202405210940}
\|\phi_x-x_*\|^2=\ld_2y''+\ld_1y'+y+\ld_1\ld_2z_1'+(\ld_1^2-2\ld_2)z_1+\ld_2^2z_2.
\end{gather}
Denote
{\footnotesize
\begin{gather}\label{202405162141}
\begin{cases}
u_{1,2}:=\frac{\kappa\be_1}{2\be_0},\\
\\
u_{1,1}:=\frac{\kappa\be_1\be_2}{2\be_0}-3-\left(\frac{2}{\kappa}-\rho\right)\be_0\ld_1\ld_2,\\
\\
u_{1,0}:=\frac{\kappa\be_1^2}{2\be_0}-2\be_2-\frac{2\be_0\ld_1^2}{\kappa}+\be_0\rho(\ld_1^2-2\ld_2),
\end{cases}
\begin{cases}
u_{2,1}:=\frac{\kappa\be_2}{2\be_0},\\
\\
u_{2,0}:=\frac{\kappa}{2\be_0}(\be_2^2-2\be_1)-\left(\frac{2}{\kappa}-\rho\right)\be_0\ld_2^2.
\end{cases}
\end{gather}
}
\noindent It follows from \eqref{202405172036}-\eqref{202405231732}, that
\begin{align}
&\nonumber-64+16(\be_2+\be_0\rho\ld_2)-4(\be_1+\be_0\rho\ld_1)+\be_0\rho\\
\label{202405171451}&\quad\quad\quad=(-64+16\be_2-4\be_1+\rho\be_0-4\rho\be_0\ld_1)+16\rho\be_0\ld_2> 0,\quad\text{(by \eqref{202405231728})}\\
\label{202405171452}&16u_{1,2}-4u_{1,1}+u_{1,0}\geq 0,\quad\text{(by \eqref{202405231729})}\\
\label{202405171453}&37-7(\be_2+\be_0\rho\ld_2)+(\be_1+\be_0\rho\ld_1)\geq 0,\quad\text{(by \eqref{202405231730})}\\
\label{202405171454}&-7u_{1,2}+u_{1,1}\geq 0,\quad\text{(by \eqref{202405231731})}\\
\label{202405171455}&-4u_{2,1}+u_{2,0}\geq 0.\quad\text{(by \eqref{202405231732})}    
\end{align}
A direct computation gives
\begin{align*}
&\quad\,\,y'''+\be_2y''+\be_1y'-3z_1'-2\be_2z_1\\
&=2\inner{x'''+\be_2x''+\be_1x'}{x-x_*}\\
&=-2\be_0\inner{U(\phi_x)}{x-x_*}\\
&=-2\be_0\inner{U(\phi_x)}{\phi_x-x_*}+2\be_0\inner{U(\phi_x)}{\ld_1x'+\ld_2x''},    
\end{align*}
which implies, by \eqref{202405210908}-\eqref{202405210909} and the Cauchy-Schwarz inequality, that
\begin{align*}
&\quad\,\, y'''+\be_2y''+\be_1y'-3z_1'-2\be_2z_1\\
&\leq-\be_0\rho\|\phi_x-x_*\|^2-\be_0\kappa\|U(\phi_x)\|^2+2\be_0\inner{U(\phi_x)}{\ld_1x'+\ld_2x''}\\
&\leq-\be_0\rho\|\phi_x-x_*\|^2-\frac{\be_0\kappa}{2}\|U(\phi_x)\|^2+\frac{2\be_0}{\kappa}\|\ld_1x'+\ld_2x''\|^2\\
&=-\be_0\rho\|\phi_x-x_*\|^2-\frac{\kappa}{2\be_0}\|x'''+\be_2x''+\be_1x'\|^2+\frac{2\be_0}{\kappa}\|\ld_1x'+\ld_2x''\|^2.    
\end{align*}
Using \eqref{202405162141}, \eqref{202303291324}, \eqref{202303291324a} and \eqref{202405210940}, the last inequality can be rewritten as
\begin{align*}
0&\geq y'''+(\be_2+\be_0\rho\ld_2)y''+(\be_1+\be_0\rho\ld_1)y'+\be_0\rho y\\
&\quad\quad\quad+u_{1,2}z_1''+u_{1,1}z_1'+u_{1,0}z_1+u_{2,1}z_2'+u_{2,0}z_2.       
\end{align*}
By multiplying both sides by $e^{4(s-t_0)}$ and then integrating $s\in[t_0,t]$, we obtain a positive constant $M_1 = M_1(t_0)$ such that
\begin{align*}
M_1&\geq e^{4(t-t_0)}[y''(t)+(\be_2+\be_0\rho\ld_2-4)y'(t)+(-4(\be_2+\be_0\rho\ld_2)+\be_1+\be_0\rho\ld_1+16)y(t)]\\
&\quad\quad\quad+[-64+16(\be_2+\be_0\rho\ld_2)-4(\be_1+\be_0\rho\ld_1)+\be_0\rho]\int\limits_{t_0}^te^{4(s-t_0)}y(s)\,ds\\
&\quad\quad\quad+e^{4(t-t_0)}[u_{1,2}z_1'(t)+(-4u_{1,2}+u_{1,1})z_1(t)]\\
&\quad\quad\quad+(16u_{1,2}-4u_{1,1}+u_{1,0})\int\limits_{t_0}^te^{4(s-t_0)}z_1(s)\,ds\\
&\quad\quad\quad+e^{4(t-t_0)}u_{2,1}z_2(t)+(-4u_{2,1}+u_{2,0})\int\limits_{t_0}^te^{4(s-t_0)}z_2(s)\,ds.    
\end{align*}
Using \eqref{202405171451}, \eqref{202405171452}, \eqref{202405171455} and then multiplying both sides by $e^{-(t-t_0)}$, we get
\begin{align*}
M_1e^{-(t-t_0)}&\geq e^{3(t-t_0)}[y''(t)+(\be_2+\be_0\rho\ld_2-4)y'(t)+(-4(\be_2+\be_0\rho\ld_2)+\be_1+\be_0\rho\ld_1+16)y(t)]\\
&\quad\quad\quad+e^{3(t-t_0)}[u_{1,2}z_1'(t)+(-4u_{1,2}+u_{1,1})z_1(t)].   
\end{align*}
Integrating the above inequality with respect to $t$ over $[t_0, \tau]$ yields
\begin{align*}
M_2&\geq e^{3(\tau-t_0)}y'(\tau)+(\be_2+\be_0\rho\ld_2-7)e^{3(\tau-t_0)}y(\tau)\\
&\quad\quad\quad+[37-7(\be_2+\be_0\rho\ld_2)+\be_1+\be_0\rho\ld_1]\int\limits_{t_0}^\tau e^{3(t-t_0)}y(t)\,dt\\
&\quad\quad\quad+u_{1,2}e^{3(\tau-t_0)}z_1(\tau)+(-7u_{1,2}+u_{1,1})\int\limits_{t_0}^\tau e^{3(t-t_0)}z_1(t)\,dt,    
\end{align*}
which implies, by \eqref{202405171453} and \eqref{202405171454}, that
\begin{gather*}
M_2\geq e^{3(\tau-t_0)}y'(\tau)+(\be_2+\be_0\rho\ld_2-7)e^{3(\tau-t_0)}y(\tau)    
\end{gather*}
or equivalently to saying that
\begin{gather*}
M_2e^{(\be_2+\be_0\rho\ld_2-10)(\tau-t_0)}\geq\frac{d}{d\tau}\left[e^{(\be_2+\be_0\rho\ld_2-7)(\tau-t_0)}y(\tau)\right].    
\end{gather*}
Integrating the above inequality with respect to $\tau$ over $[t_0, T]$ and then dividing both sides by $e^{(\be_2+\be_0\rho\ld_2-7)(T-t_0)}$, we obtain
\begin{align*}
y(T)&\leq e^{-(\be_2+\be_0\rho\ld_2-7)(T-t_0)}\int\limits_{t_0}^TM_2e^{(\be_2+\be_0\rho\ld_2-10)(\tau-t_0)}\,d\tau\\
&\quad\quad\quad+e^{-(\be_2+\be_0\rho\ld_2-7)(T-t_0)}y(t_0)=\mathcal{O}(e^{-3T}).    
\end{align*}
\end{proof}

In the case when $\ld_1=\ld_2=0$, conditions \eqref{202405231728}-\eqref{202405231732} can be omitted.
\begin{cor}
Suppose $\lambda_1 = \lambda_2 = 0$ and the conditions in Assumption \ref{asu-202405161507a}. Then the trajectory $x(\cdot)$ generated by the dynamical system \eqref{ode3} converges strongly to the unique solution $x_*$ at the rate
\begin{align*}
\|x(t)-x_*\|^2=\mathcal{O}(e^{-3t}),
\end{align*}
provided conditions \eqref{202405172036}-\eqref{202405172038} hold.
\end{cor}

\begin{exa}\label{examStrongMonotone}
We consider the same setting as in Example \ref{examMonotone}, except that we take $U(x) := Ax$, where $A \in \mathbb{R}^{n\times n}$ is a random symmetric positive definite matrix. Then $U$ is strongly monotone with $\rho=\lambda_{\min}$ and Lipschitz continuous with $L_u=\lambda_{\max}$, where $\lambda_{\min}, \lambda_{\max}$ are the minimum and maximum eigenvalues of $A$, respectively. Problem \eqref{pb-zer} has a unique solution $x_* = (0,0,...,0)$. We consider three dynamical systems with initial conditions as in Example \ref{examMonotone}, but now for the third-order dynamical system we take
\begin{align*}
 & \lambda_1 = \lambda_2 = 0,\\
  & \be_2 = \frac{36}{\rho\kappa}+ 10,\\
  & \be_1 = 0.5 \left[(\be_2-5)\frac{16}{\rho\kappa} + \frac{\be_2(\be_2-4)}{2} \right],\\
&\be_0 = 0.5 \left[ \frac{4}{\rho}(-4\be_2+\be_1+16) + \frac{\kappa\be_1}{2}\cdot\frac{-4\be_2+\be_1+16}{2\be_2-12}\right],
\end{align*}
so that conditions \eqref{202405172036}-\eqref{202405172038} are satisfied. \\

The error $\|x(t)-x_*\|^2$ displayed in Figure \ref{figExamStrongmono} (with $50$ random samples of $A$) shows that the third-order dynamical system achieves the theoretical  rate $\mathcal{O}(e^{-3t})$. 
Also, it clearly shows that the trajectories generated by the third-order dynamical system converge faster than those of the first- and second-order dynamical systems.

\begin{figure}[h]
    \centering
\includegraphics[width=0.48\textwidth]{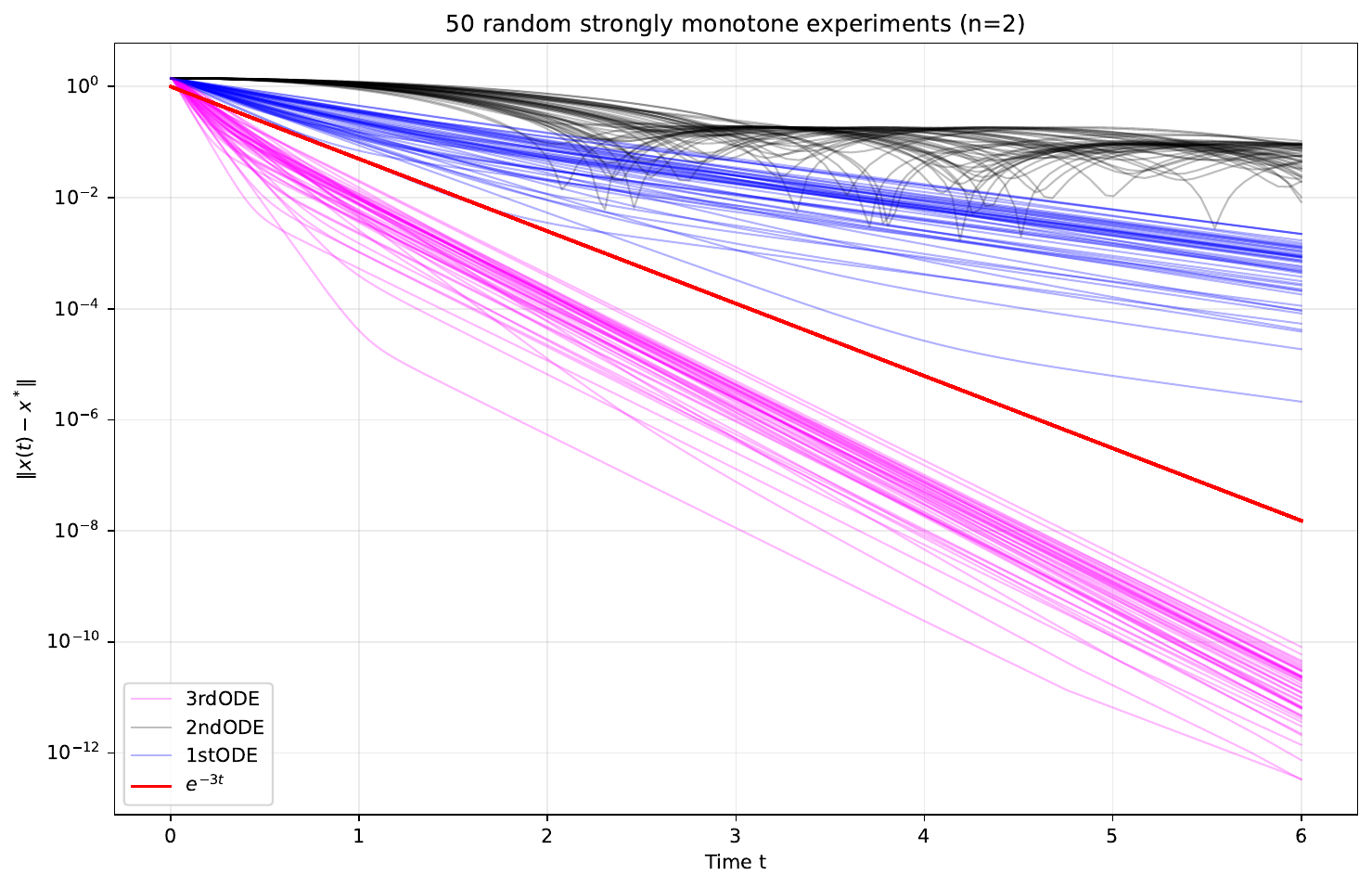} 
\includegraphics[width=0.48\textwidth]{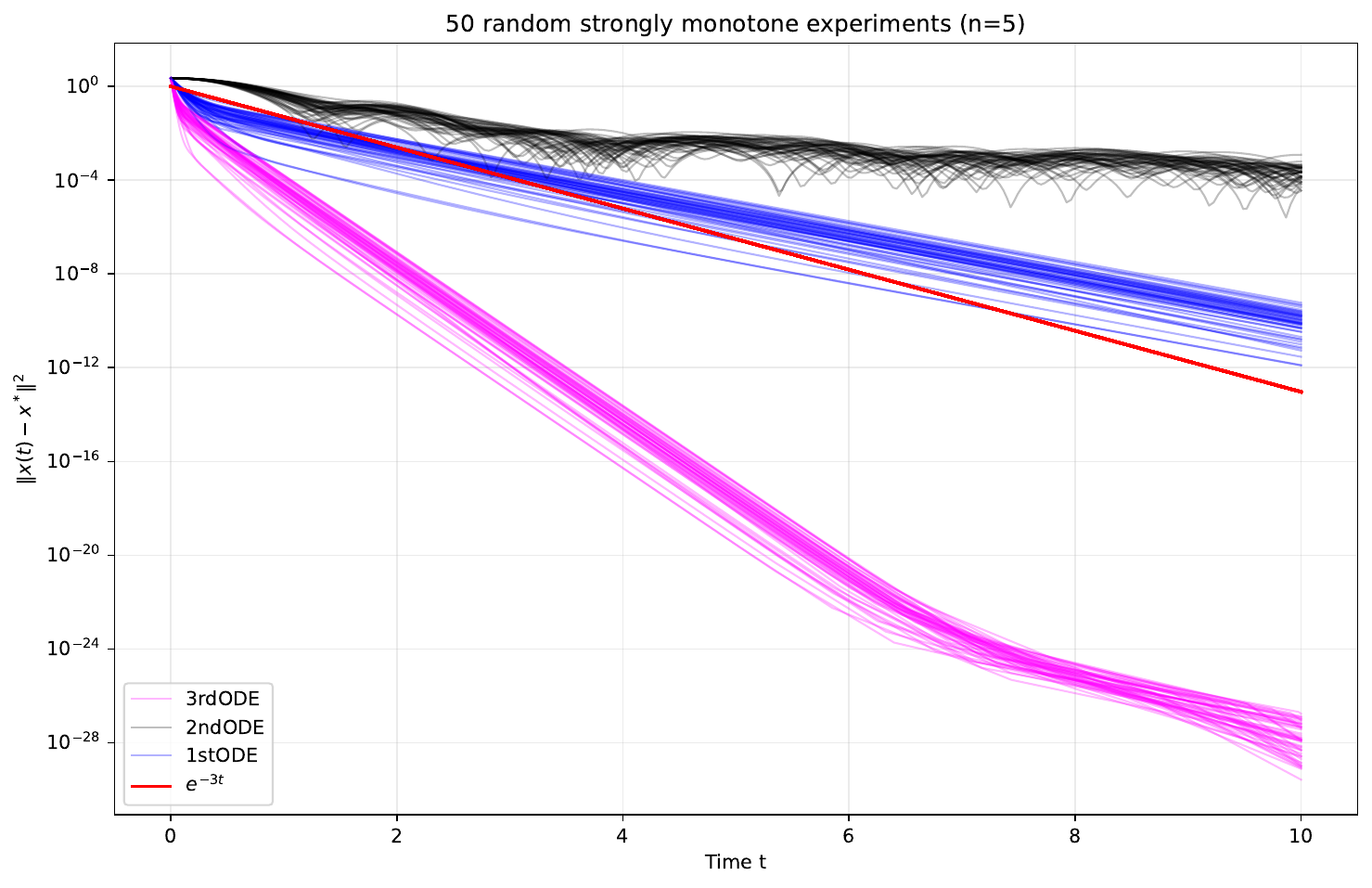} 
    \caption{The error $\|x(t)-x^*\|^2$  in Example \ref{examStrongMonotone} with $n=2$ (left) and  $n=5$ (right) for 50 random samples.}
    \label{figExamStrongmono}
\end{figure}
\end{exa}

As an illustration,  we apply Theorem \ref{thm20240523} to  the variational inequality \eqref{proA+N1} recalled below in an equivalent form:
\begin{gather}\label{VIVome}\tag{VI($V,\ome$)}
\text{Find $x_*\in\ome\subseteq\calh$ such that $\inner{V(x_*)}{y-x_*}\geq 0$ for all $y\in\ome$.}
\end{gather}
The solution set of Problem \ref{VIVome} is denoted as $S(V,\ome)$.

\begin{asu}\label{asu-202405161507}
Assume that
\begin{enumerate}[label=(\roman*)]
\item The set $\ome\subseteq\calh$ is nonempty, closed and convex.
    \item The operator $V:\calh\to\calh$ is $\ell$-strongly pseudo-monotone.
    \item The operator $V:\calh\to\calh$ is $M$-Lipschitz continuous.
    \item Let $\nu$ be the parameter such that
    \begin{gather}
0<\nu<\frac{4\ell}{M^2}.
    \end{gather}
\end{enumerate}
\end{asu}

\begin{rem}
It is well-known that the solution set $S(V,\ome)$ is a singleton provided that the operator $V:\calh\to\calh$ is $\ell$-strongly pseudo-monotone \cite{zbMATH06697598}.
\end{rem}

\begin{prop}[{\cite[Proposition 2.6]{zbMATH07389165}}]
Let $\overline{v}\in\calh$ and $x_*\in S(V,\ome)$. Under Assumption \ref{asu-202405161507}, it holds that
\begin{align}
\label{202405162044}&\inner{\overline{v}-\Pi_\ome(\overline{v}-\nu V(\overline{v}))}{\overline{v}-x_*}\geq\kappa_1\|\overline{v}-\Pi_\ome(\overline{v}-\nu V(\overline{v}))\|^2,\\
\label{202405162045}&\|\overline{v}-\Pi_\ome(\overline{v}-\nu V(\overline{v}))\|\geq\kappa_2\|\overline{v}-x_*\|,    
\end{align}
where
\begin{gather}
\label{202405162046}\kappa_1:=1-\frac{\nu M^2}{4\ell},\quad\kappa_2:=\frac{\nu\ell}{1+\nu\ell+\nu M}.
\end{gather}
and  $\Pi_\ome$ stands for the projection operator onto $\ome$.
\end{prop}

We associate with Problem \ref{VIVome} the dynamical system \eqref{ode3}, in which the coefficients $\be_0,\be_1,\be_2,\ld_1,\ld_2$ are positive constants and the operator $U:=I-\Pi_\ome(I-\nu V)$. In detail, the dynamical system takes the form
\begin{gather}\label{ode3-VI}
x'''(t)+\be_2x''(t)+\be_1x'(t)+\be_0[\phi_x(t)-\Pi_\ome(\phi_x(t)-\nu V(\phi_x(t)))]=0,
\end{gather}
where
\begin{gather*}
\phi_x(t)=x(t)+\ld_1x'(t)+\ld_2x''(t).
\end{gather*}

Under suitable assumptions, the dynamical system \eqref{ode3-VI} offers the convergence rate $\mathcal{O}(e^{-3t})$, which is better than the rate $\mathcal{O}(e^{-2t})$ obtained by Hai and Vuong in \cite{HaiVuong24} for the case when $\ld_1=\ld_2=0$.
\begin{cor}
Consider the dynamical system \eqref{ode3-VI}, under Assumption \ref{asu-202405161507}. Then the trajectory $x(\cdot)$ generated by the dynamical system \eqref{ode3-VI} converges strongly to the unique solution $x_*$ at the rate
\begin{align*}
\|x(t)-x_*\|^2=\mathcal{O}(e^{-3t}),
\end{align*}
provided conditions \eqref{202405172036}-\eqref{202405231732} hold, where $\kappa_1,\kappa_2$ are the constants stated in \eqref{202405162044}-\eqref{202405162046} and
\begin{gather*}
\kappa:=\kappa_1^3\kappa_2^2,\quad\rho:=\kappa_1\kappa_2^2.
\end{gather*}
\end{cor}
\begin{proof}
By \eqref{202405162044}-\eqref{202405162045}, we have
\begin{align*}
\inner{U(\overline{v})}{\overline{v}-x_*}\geq\kappa_1\|U(\overline{v})\|^2\geq\kappa_1\kappa_2^2\|\overline{v}-x_*\|^2=\rho\|\overline{v}-x_*\|^2,
\end{align*}
and so the operator $U$ satisfies Assumption \ref{asu-202405161507a}(i). It follows from \eqref{202405162044} that 
\begin{align*}
\kappa_1\|U(\overline{v})\|^2\leq\inner{U(\overline{v})}{\overline{v}-x_*}\leq\|U(\overline{v})\|\cdot\|\overline{v}-x_*\|.
\end{align*}
A division by $\|U(\overline{v})\|$ produces $\kappa_1\|U(\overline{v})\|\leq \|\overline{v}-x_*\|$, confirming that the operator fulfills Assumption (\ref{asu-202405161507a}(ii) with Lipschitz constant $1/\kappa_1$. Theorem \ref{thm20240523} can then be applied by setting the constant
\begin{align*}
\kappa=\frac{\rho}{L_u^2}=\kappa_1^3\kappa_2^2.
\end{align*}
\end{proof}

In Remark \ref{rem-compare}, we highlight how time scaling distinctively affects the trajectories of the second- and third-order systems.
\begin{rem}\label{rem-compare}
While the exponential convergence rate of the second-order equation $x''(t) + a_1(t)x'(t) = -a_0(t)B(x(t))$ can be improved from $\mathcal{O}(e^{-t})$ to $\mathcal{O}(e^{-ct})$ by utilizing a time rescaling transformation $y(t) = x(ct)$ with $c>0$, its capacity for acceleration cannot match that of higher-order systems. As demonstrated in Theorem \ref{thm20240523}, the third-order equation inherently converges at a rate of $\mathcal{O}(e^{-3t})$. When the same time rescaling is applied, it reaches an exceptionally fast rate of $\mathcal{O}(e^{-3ct})$. Since $e^{-3ct}<e^{-ct}$, the third-order system exhibits strictly better convergence properties than the second-order formulation, maintaining its competitive superiority regardless of whether time rescaling is employed.
\end{rem}

\section{Splitting monotone inclusions} \label{Mon}
In this section, we discuss the applications of  the proposed third-order ODE to the following three operators splitting inclusion
\begin{gather} \label{proA+B2}
    \text{Find $x_*\in \calh$ such that}\quad 0\in A(x_*)+B(x_*)+C(x_*),
\end{gather}
where $A:\calh\rightarrow\calh$ is a single-valued operator and 
$B, C:\calh \rightrightarrows\calh$ are maximal monotone set-valued operators defined on $\calh$.
The zero set of \eqref{proA+B2} is denoted by $Z(A+B+C)$. This monotone inclusion provides a very general framework and has many applications in optimization \cite{zbMATH06834688}. \\

{\bf Davis-Yin operator:}
When $A$ is $\om_A$ cocoercive, Davis and Yin introduced in \cite{zbMATH06834688} the following operator
$$
T_{DY}: \calh \to \calh, \quad 
T_{DY} := J_{\gamma B} \circ (2 J_{\gamma C} -I -\gamma A \circ J_{\gamma C}) + I -J_{\gamma C},
$$
where $\gamma \in (0, 2 \om_A)$. Then the zero set $Z(A+B+C)$ can be characterized in terms of $T_{DY}$ by
$$
Z(A+B+C) = J_{\gamma C}({\text{Fix}}(T_{DY})),
$$
where ${\text{Fix}}(T_{DY})$ stands for the set of fixed points of $T_{DY}$. 
Now let us define the operator $U: \calh \to \calh$ by
$$
U:= I-T_{DY} = J_{\gamma C}- J_{\gamma B} \circ (2 J_{\gamma C} -I -\gamma A \circ J_{\gamma C})
$$
then it holds
 $$
Z(A+B+C) = J_{\gamma C}({\text{Fix}}(T_{DY}))
= J_{\gamma C} (Z(U)).
$$
Hence, we can apply the convergence results for the third-order ODEs established in the previous sections to solve Problem \eqref{proA+B2}. Indeed, we can verify this condition in the following Lemma.
\begin{lem}
  For all $\gamma \in (0, 2 \om_A)$, the operator $U$ is cocoercive and hence demiclosed at $0$.   
\end{lem}
\begin{proof}
    It follows from \cite[Proposition 2.1]{zbMATH06834688} that $T_{DY}$ is $\alpha$-averaged with coefficient
       $$\alpha := \frac{2 \om_A}{4 \om_A -\gamma} \in (0,1).$$
     Therefore, $U=I-T_{DY}$ is $\frac{1}{2\alpha}$-cocoercive \cite[Remark 3.2 and Corollary 3.5]{zbMATH07402653}. 
\end{proof}
\begin{rem}
If $C=0$, then the monotone inclusion \eqref{proA+B2} reduces to the two-operator splitting monotone inclusion \eqref{proA+B}. In this case, the $T_{DY}$ is nothing but the forward-backward operator, i.e.  
$T_{DY} = J_{\gamma B} (I-\gamma A) = T_{FB}$. 
\\
If $A=0$ then \eqref{proA+B2} reduces to 
\begin{gather} \label{proB+C}
    \text{Find $x_*\in \calh$ such that}\quad 0\in B(x_*)+C(x_*),
\end{gather}
where  
$B, C:\calh \rightrightarrows\calh$ are maximal monotone set-valued operators. In this case, $T_{DY}$ coincides with the Douglas-Rachford operator studied in \cite{zbMATH03664304}
$$
T_{DY}= J_{\gamma B} \circ (2 J_{\gamma C} -I) + I - J_{\gamma C} =T_{DR}.
$$
\end{rem}

{\bf Forward-backward-half forward operator:}
We consider the case when $A$ is $\om_A$ cocoercive and $C$ is single-valued, merely monotone and $L$-Lipschitz continuous.  We  will employ the forward-backward-half forward (FBHF) operator $T:\calh \to \calh$ studied in \cite{FBHF}
$$
T = (I-\gamma C)J_{\gamma B}(I - \gamma (A+C)) + \gamma C,
$$
for some $\gamma >0$. 
This operator is very general, which includes both forward-backward $T_{FB}$ and forward-backward-forward operator $T_{FBF}$ proposed by Tseng \cite{zbMATH01416955} as a particular case. Indeed, if $C=0$ then $T$ reduces to $T_{FB}$, and if $A=0$ then $T$ reduces to $T_{FBF}$ defined as
\begin{equation}
    T_{FBF} = (I-\gamma C)J_{\gamma B}(I - \gamma C) + \gamma C.
\end{equation}
We have the following Lemma.
\begin{lem}\cite[Proposition 2.1]{FBHF} \label{fbhf}
Assume that $A$ is $\om_A$-cocoercive, $C$ is monotone and $L$-Lipschitz continuous, and $B$ is maximal monotone. Then for all $\gamma < \frac{4 \om_A}{1+\sqrt{1+16 \om_A^2 L^2}}$, it holds
\begin{itemize}
\item[(i)] $Z(A+B+C) = {\text{Fix}} \, (T)$;
\item[(ii)] $T$ is quasi-nonexpansive, i.e.
$$
\|T (x) - T(x^*)\| \le \|x-x^*\| \quad \forall x \in \calh, \forall x^* \in \text{Fix} \,(T)
$$
\end{itemize}
\end{lem}
Let us define the operator $U: \calh \to \calh$ by
$
U:= I- T,
$
then we have 
$$
Z(A+B+C) = {\text{Fix}} \, (T) = Z (U).
$$
\begin{lem}
    Under conditions made in Lemma \ref{fbhf}, the operator $U$ is quasi-cocoercive with modulus $\om =\frac{1}{2}$.
\end{lem}
\begin{proof}
    Let $x \in \calh$ and $x^* \in Z (U)$, we have
    $x^* = T (x^*)$. Since $T$ is quasi-nonexpansive, we can deduce
    \begin{align*}
        \inner{U(x)}{x-x^*} 
        & = \|U(x)\|^2 + \inner{U(x)}{x-U(x)-x^*}\\
        & = \|U(x)\|^2 + \inner{(I-T)(x)}{T(x)-T(x^*)}\\
        & = \|U(x)\|^2 + \frac{1}{2} \left( \|x-x^*\|^2 - \|T(x) - T(x^*)\|^2 - \|(I-T)(x)\|^2 \right)\\
        & \ge  \|U(x)\|^2 - \frac{1}{2}\|(I-T)(x)\|^2\\
        & = \frac{1}{2}\|U(x)\|^2.
    \end{align*}
\end{proof}
The next lemma shows that even though the operator $U$ is only quasi-cocoercive, it is still demiclosed at $0$. Hence, the convergence results of the third-order ODE discussed previously for solving \eqref{proA+B} also follow.
\begin{lem}
Let $A:\calh\to\calh$ be $\omega_A$-cocoercive, $C:\calh\to\calh$ monotone and $L$-Lipschitz
continuous, and $B:\calh\rightrightarrows\calh$ maximal monotone. Let
$0<\gamma<\frac{4 \om_A}{1+\sqrt{1+16 \om_A^2 L^2}}$ and
\[
T:=(I-\gamma C)J_{\gamma B}\bigl(I-\gamma(A+C)\bigr)+\gamma C,
\qquad U:=I-T .
\]
Then $U$ is demiclosed at $0$, i.e. if $x_n\rightharpoonup x$ and $Ux_n\to 0$, then $Ux=0$.
\end{lem}

\begin{proof}
Recall that $Z(A+B+C)=Fix T=Z(U)$, we will use 
the maximal monotonicity of $A+B+C$.
Let $x_n\rightharpoonup x$ and
$Ux_n:=x_n-Tx_n\to 0$. Set
\[
y_n:=J_{\gamma B}\bigl(x_n-\gamma(A+C)x_n\bigr).
\]
By definition of the resolvent,
\begin{equation}\label{eq:incl}
\frac{x_n-y_n}{\gamma}-(A+C)x_n\in By_n .
\end{equation}
Moreover $Tx_n=(I-\gamma C)y_n+\gamma Cx_n=y_n-\gamma(Cy_n-Cx_n)$, hence
\[
Ux_n=(x_n-y_n)+\gamma(Cy_n-Cx_n).
\]
By the Lipschitz continuity of $C$,
\[
\norm{x_n-y_n}\le\norm{Ux_n}+\gamma L\norm{y_n-x_n}.
\]
Since $\gamma L<1$, we get $(1-\gamma L)\norm{x_n-y_n}\le\norm{Ux_n}\to0$.
Thus $x_n-y_n\to0$ strongly, and consequently $y_n\rightharpoonup x$.

Adding $(A+C)y_n$ to both sides of \eqref{eq:incl} gives
\[
v_n:=\frac{x_n-y_n}{\gamma}+(A+C)y_n-(A+C)x_n\in (A+B+C)(y_n).
\]
Since $A$ is $\omega_A$-cocoercive, it is $\tfrac1{\omega_A}$-Lipschitz and  $C$
is $L$-Lipschitz we decude
\[
\norm{v_n}\le\Bigl(\frac1\gamma+\frac1{\omega_A}+L\Bigr)\norm{x_n-y_n} \to 0.
\]
Since $A+C$ is monotone and continuous with full domain, and $B$
is maximal monotone, the sum
$A+B+C$ is maximal monotone. Since the graph of a 
maximally monotone is sequentially closed corresponding to the weak-strong topology of the product $\calh\times\calh$, we deduce 
 $x\in Z(A+B+C) = Z(U)$. 
\end{proof}

We close this section by mentioning a particular application of monotone inclusion arising in saddle-point problem of convex-concave function:
$$
\min_{x \in \calh} \max_{y \in \calh} \, f(x) + L(x,y)-g(y)
$$
where $L$ is a convex-concave smooth function with Lipschitz continuous gradients and $f,g: \calh \to \mathbb{R}$ are are proper, lsc, convex functions. The optimality condition of the above saddle-point problem reduces to a monotone inclusion in product spaces: Find
$(x^*,y^*) \in \calh \times \calh$ such that
$$
0 \in \begin{pmatrix}
    \nabla_x L(x^*,y^*)\\ -\nabla_y L(x^*,y^*)
\end{pmatrix} 
+ \begin{pmatrix}
    \partial f(x^*)\\ \partial g(y^*)
\end{pmatrix},
$$
where $\partial f, \partial g$ denote the subdifferential of convex functions $f,g$, respectively.
This problem can be treated as a particular case of monotone inclusion \eqref{proA+B2} with $A=0, B= [\partial f; \partial g]$ and $C=[\nabla_x L; - \nabla_y L]$. In this case, $B$ is maximal monotone and $C$ is monotone and Lipschitz continuous, hence we can handle this saddle-point problem using the FBF and/or FBHF operators.

\section{Fast rate of $\mathcal{O}(\frac{1}{t^3})$ for convex optimization}\label{Optim}

It is well-known that when the operator $U$ is the gradient of a convex function $f$ (i.e., $U = \nabla f$), Problem \eqref{pb-zer} reduces to finding the stationary points of $f$, which is equivalent to solving the convex minimization problem \eqref{Op}. The dynamical system \eqref{ode3}, where $U=\nabla f$, takes the following form
\begin{gather}\label{202404190913}
x'''(t)+\be_2(t)x''(t)+\be_1(t)x'(t)+\be_0(t)\nabla f(\phi_x(t))=0,
\end{gather}
where
\begin{gather}
\phi_x(t)=x(t)+\ld_1(t)x'(t)+\ld_2(t)x''(t).
\end{gather}

In this section, we assume that the function $f(x)$ attains its minimum value $f_* := \min_{x \in \mathcal{H}} f(x)$ at least at one global minimizer $x_* \in \mathcal{H}$. We study the dynamical system \eqref{202404190913} whose coefficients take the following forms
\begin{gather}\label{202404282157}
\begin{cases}
\ld_2(t)=\xi_2t^2,\quad\ld_1(t)=\xi_1t,\\
\\
\be_2(t)=\frac{\nu_2}{t},\quad\be_1(t)=\frac{\nu_1}{t^2},\quad\be_0(t)\equiv\al_0.
\end{cases}
\end{gather}

Substituting \eqref{202404282157} back into \eqref{202404190913}, we get
\begin{gather}\label{ode-1/t^3}
x'''(t)+\frac{\nu_2}{t}x''(t)+\frac{\nu_1}{t^2}x'(t)=-\al_0\nabla f\left(x(t)+\xi_1tx'(t)+\xi_2t^2x''(t)\right).   
\end{gather}

Proposition \ref{prop-202511021525} proves the existence of a global solution to equation \eqref{ode3}, assuming that the operator $U$ satisfies a global Lipschitz condition; that is, it is Lipschitz continuous on the entire space $\calh$. In the special case where $U$ is the gradient of a convex function $f(x)$, this assumption can be considerably weakened to local Lipschitz continuity on bounded subsets of $\calh$. The proof follows a natural two-step strategy: first, constructing a local solution using arguments parallel to Proposition \ref{prop-202511021525}, and second, ruling out finite-time blow-up to extend the solution globally.
\begin{prop}\label{prop-20261018}
Consider the dynamical system \eqref{ode-1/t^3}, where $f:\calh\to\R$ is a convex function whose gradient is Lipscitz continuous on bounded sets and $\text{arg\,min}(f)\ne\emptyset$. Then for $(x_0,x_1,x_2)\in\calh\times\calh\times\calh$ and $t_0>0$ there exists a unique strong global solution $x(t)$ of \eqref{ode-1/t^3} with the initial condition $x(t_0)=x_0$, $x'(t_0)=x_1$, $x''(t_0)=x_2$.
\end{prop}
\begin{proof}
See the Appendix.
\end{proof}

We now investigate the convergence properties of the dynamical system \eqref{ode-1/t^3} in two cases.

\subsection{The first case where $\xi_2> 0$}
\begin{thm}\label{thm51}
Suppose that the function $f:\calh\to\R$ is convex and differentiable. Consider the dynamical system \eqref{ode-1/t^3}, where $\al_0>0$, $t_0>0$ and
\begin{gather*}
\begin{cases}
\xi_2>0,\\
\\
\nu_1=\frac{\xi_1+1}{\xi_2},\quad\nu_2=\frac{\xi_1}{\xi_2}+2.
\end{cases}
\end{gather*}
The following assertions hold.
\begin{enumerate}
\item[(i)] It holds that
\begin{gather*}
f\left(x(t)+\xi_1tx'(t)+\xi_2t^2x''(t)\right)-f_*=\mathcal{O}\left(\frac{1}{t^3}\right)\quad\forall t\geq t_0.
\end{gather*}
\item[(ii)] If the parameters $\xi_1,\xi_2$ satisfy
\begin{align}
\label{2024043020577}&\xi_2<\frac{1}{9},\\
\label{202404302058}\xi_2+2\sqrt{\xi_2}<\,\,&\xi_1<4\xi_2+\frac{1}{3},
\end{align}
then we have
\begin{gather}
f(x(t))-f_*=\mathcal{O}\left(\frac{1}{t^3}\right).
\end{gather}
\end{enumerate}
\end{thm}
\begin{proof}
Note that
\begin{align*}
\nonumber\phi_x'(t)&=(1+\xi_1)x'(t)+(\xi_1+2\xi_2)tx''(t)+\xi_2t^2x'''(t)\\
\label{202404290929}&=-\al_0\xi_2t^2\nabla f(\phi_x(t))\quad\text{(by \eqref{ode-1/t^3})}.    
\end{align*}
(i) Let us define the function
\begin{gather*}
V(t):=\frac{\al_0t^3}{3\xi_2}[f(\phi_x(t))-f_*]+\frac{1}{2}\|t^2x''(t)+\frac{\xi_1}{\xi_2}tx'(t)+(\nu_1-\nu_2+2)(x(t)-x_*)\|^2.    
\end{gather*}  
We have
\begin{gather*}
V'(t)=\frac{\al_0t^2}{\xi_2}[f(\phi_x(t))-f_*]+\frac{\al_0t^3}{3\xi_2}\inner{\nabla f(\phi_x(t))}{\phi_x'(t)}\\
+\inner{t^2x'''(t)+\nu_2tx''(t)+\nu_1x'(t)}{t^2x''(t)+\frac{\xi_1}{\xi_2}tx'(t)+(\nu_1-\nu_2+2)(x(t)-x_*)},
\end{gather*}
which implies, by \eqref{ode-1/t^3} and \eqref{202404290929}, that
\begin{align*}
V'(t)&=\frac{\al_0t^2}{\xi_2}[f(\phi_x(t))-f_*]-\frac{\al_0^2t^5}{3}\|\nabla f(\phi_x(t))\|^2\\
&\quad\quad\quad-\al_0t^4\inner{\nabla f(\phi_x(t))}{x''(t)}-\al_0\frac{\xi_1}{\xi_2}t^3\inner{\nabla f(\phi_x(t))}{x'(t)}\\
&\quad\quad\quad-\al_0(\nu_1-\nu_2+2)t^2\inner{\nabla f(\phi_x(t))}{x(t)-x_*}.    
\end{align*}
Since
\begin{gather*}
\inner{\nabla f(\phi_x(t))}{x(t)-x_*}=\inner{\nabla f(\phi_x(t))}{\phi_x(t)-x_*}-\inner{\nabla f(\phi_x(t))}{\xi_1tx'(t)+\xi_2t^2x''(t)},
\end{gather*}
we get
\begin{align*}
V'(t)&=\frac{\al_0t^2}{\xi_2}[f(\phi_x(t))-f_*]-\frac{\al_0^2t^5}{3}\|\nabla f(\phi_x(t))\|^2\\
&\quad\quad\quad-\al_0(\nu_1-\nu_2+2)t^2\inner{\nabla f(\phi_x(t))}{\phi_x(t)-x_*}\\
&\leq\al_0\left(\frac{1}{\xi_2}-(\nu_1-\nu_2+2)\right)t^2[f(\phi_x(t))-f_*]=0.    
\end{align*}
Thus, we get
\begin{gather*}
V(t_0)\geq V(t)\geq\frac{\al_0t^3}{3\xi_2}[f(\phi_x(t))-f_*]. 
\end{gather*}

(ii) Consider the quadratic function
\begin{gather*}
h(z)=z^2-(\xi_1-\xi_2)z+\xi_2.
\end{gather*}
By the first inequality in \eqref{202404302058}, namely $\xi_2 + 2\sqrt{\xi_2} < \xi_1$, the quadratic function $h(z)$ has two distinct positive roots, denoted by $p$ and $q$. Again using \eqref{202404302058}, we have
\begin{align*}
p+q=\xi_1-\xi_2<3\xi_2+\frac{1}{3}<\frac{2}{3}\quad\text{(by \eqref{2024043020577})}
\end{align*}
and 
\begin{align*}
h\left(\frac{1}{3}\right)=\frac{1}{9}-\frac{\xi_1}{3}+\frac{4\xi_2}{3}>0\quad\text{(by \eqref{202404302058})}. 
\end{align*}
Hence, $0<p,q<\frac{1}{3}$. Now we can write $\phi_x$ in the following form
\begin{gather*}
\phi_x(t)=(x(t)+qtx'(t))+pt(x(t)+qtx'(t))'.
\end{gather*}
Applying Lemma \ref{lem202404302050} successively yields
\begin{gather*}
f(x(t)+qtx'(t))-f_*=\mathcal{O}\left(\frac{1}{t^3}\right),\quad f(x(t))-f_*=\mathcal{O}\left(\frac{1}{t^3}\right).
\end{gather*}
\end{proof}

\begin{rem}\label{rem-62}
    To highlight the acceleration achieved, it is instructive to compare our framework with the second-order system $x''(t) + \frac{3}{t}x'(t) = -\nabla f(x(t))$. The baseline convergence rate $f(x(t))-f_*$ for this second-order model is $\mathcal{O}(\frac{1}{t^2})$. If one employs a time rescaling $t \mapsto t^2$ by performing the change of variables $y(t) = x(t^2)$, the system transforms into a second-order ODE with $y(t)$ as the unknown function, leading to a scaled rate of $\mathcal{O}(\frac{1}{t^4})$. In comparison, Theorem \ref{thm51} offers an inherent rate of $\mathcal{O}(\frac{1}{t^3})$ for the third-order equation. Applying the same time transformation to our third-order system further accelerates its rate to $\mathcal{O}(\frac{1}{t^6})$. Crucially, this demonstrates that the third-order systems strictly dominate the second-order systems in terms of convergence speed, both in the unscaled and time-scaled regimes.
\end{rem}

\subsection{The second case where $\xi_2=0$}
\begin{thm}\label{thm20240505}
Suppose that the function $f:\calh\to\R$ is convex and differentiable. Consider the dynamical system \eqref{ode-1/t^3}, where $\al_0>0$, $t_0>0$, $\xi_1>0$, $\xi_1\ne\frac{1}{3}$ and
\begin{gather*}
\nu_1=\left(\frac{1+\xi_1}{\xi_1}\right)(\nu_2-2)-\frac{1+\xi_1}{\xi_1^2}.
\end{gather*}
Suppose that
\begin{gather}\label{202404301445}
\nu_2\geq 6+\frac{1}{\xi_1}.
\end{gather}
Then it holds that
\begin{gather}
\label{202404301413}f\left(x(t)+\xi_1tx'(t)\right)-f_*=\mathcal{O}\left(\frac{1}{t^3}\right),\\
\label{202404301414}f(x(t))-f_*\leq\frac{M_1}{t^3}+\frac{M_2}{t^{\frac{1}{\xi_1}}},
\end{gather}
for some $M_1, M_2 > 0$.
\end{thm}
\begin{proof}
It follows from \eqref{202404301445}, that
\begin{gather}\label{202404301449}
\frac{3}{\xi_1}\leq\frac{1}{\xi_1}(\nu_2-2)-\frac{1+\xi_1}{\xi_1^2}.
\end{gather}
To prove \eqref{202404301413}, we consider the function
\begin{align*}
V(t)&:=\frac{\al_0t^3}{\xi_1}[f(\phi_x(t))-f_*]+\frac{1}{2}\|t^2x''(t)+(\nu_2-2)tx'(t)\\
&\quad\quad\quad+\left(\frac{1}{\xi_1}(\nu_2-2)-\frac{1+\xi_1}{\xi_1^2}\right)(x(t)-x_*)\|^2.    
\end{align*}
Using \eqref{ode-1/t^3}, where $\xi_2=0$, we have
\begin{gather*}
V'(t)=\frac{3\al_0t^2}{\xi_1}[f(\phi_x(t))-f_*]+\frac{\al_0t^3}{\xi_1}\inner{\nabla f(\phi_x(t))}{\phi_x'(t)}\\
+\inner{-\al_0t^2\nabla f(\phi_x(t))}{t^2x''(t)+(\nu_2-2)tx'(t)+\left(\frac{1}{\xi_1}(\nu_2-2)-\frac{1+\xi_1}{\xi_1^2}\right)(x(t)-x_*)}.
\end{gather*}
Since $\phi_x'=(1+\xi_1)x'+\xi_1tx''$, the line above becomes
\begin{align*}
V'(t)&=\frac{3\al_0t^2}{\xi_1}[f(\phi_x(t))-f_*]+\left(\frac{1+\xi_1}{\xi_1}-(\nu_2-2)\right)\al_0t^3\inner{\nabla f(\phi_x(t))}{x'(t)}\\
&\quad\quad\quad-\al_0\left(\frac{1}{\xi_1}(\nu_2-2)-\frac{1+\xi_1}{\xi_1^2}\right)t^2\inner{\nabla f(\phi_x(t))}{x(t)-x_*}\\
&=\frac{3\al_0t^2}{\xi_1}[f(\phi_x(t))-f_*]-\al_0\left(\frac{1}{\xi_1}(\nu_2-2)-\frac{1+\xi_1}{\xi_1^2}\right)t^2\inner{\nabla f(\phi_x(t))}{\phi_x(t)-x_*}.    
\end{align*}
Using the convexity of the function $f$, we get
\begin{align*}
V'(t)&\leq\left(\frac{3}{\xi_1}-\frac{1}{\xi_1}(\nu_2-2)+\frac{1+\xi_1}{\xi_1^2}\right)\al_0t^2[f(\phi_x(t))-f_*]\\
&\leq 0\quad\text{(by \eqref{202404301449})}.    
\end{align*}
Consequently,
\begin{gather*}
V(t_0)\geq V(t)\geq\frac{\al_0t^3}{\xi_1}[f(\phi_x(t))-f_*].
\end{gather*}

To show \eqref{202404301414}, we make use of Lemma \ref{lem202404302050} for the function $g(z):=f(z)-f_*$.
\end{proof}

\begin{rem}
In the case when $\xi_1=\frac{1}{4}$, Theorem \ref{thm20240505} reduces to the results obtained by Attouch et al obtained in \cite{zbMATH07814972}. Much like the observations in Remark \ref{rem-62}, the third-order model studied here inherently outperforms the second-order model in convergence speed, with or without a time rescaling $t \mapsto t^2$ via doing the change of variables $y(t) = x(t^2)$.
\end{rem}

\section{Conclusion}
In this paper, we introduce  and analyze the convergence properties of a general third-order dynamical system designed for solving generalized monotone equations. We establish theoretical foundations including solution existence, asymptotic convergence, and rapid exponential convergence under appropriate parameter conditions.
We further demonstrate the practical applications of our approach to variational inequalities, monotone inclusions, and convex optimization problems. 
For future research, we intend to enhance the third-order dynamical system by incorporating time-rescaling techniques and investigate exponential convergence under weaker regularity conditions. Specifically, we aim to replace the traditional (quasi)-strong monotonicity assumptions with more general conditions such as metric regularity and/or error bound conditions, thereby broadening the scope of applicability.

\section*{Acknowledgments} The research of P.V. Hai is funded by Vietnam National Foundation for Science and Technology Development (NAFOSTED) under grant number 101.02-2025.08.

\section{Statements and Declarations}
\subsection{Funding}
The research of P.V. Hai is funded by Vietnam National Foundation for Science and Technology Development (NAFOSTED) under grant number 101.02-2025.08

\subsection{Competing Interests}
There are no conflicts of interest or competing interests related to this manuscript.

\subsection{Author Contributions}

 All authors contributed equally to the study conception, design and implementation, wrote and corrected the manuscript.

\subsection{Availability of supporting data}
No data sets were generated during the current study.



\section{Appendix}
\subsection{Proof of Lemma \ref{lem202404302050}}
\begin{proof}
    Let $t,s\geq t_0>0$, $\varepsilon>0$ and $\phi_y(t)=y(t)+\xi ty'(t)$. We have
\begin{gather*}
\frac{1}{\xi}t^{\frac{1}{\xi}-1}\phi_y(t)=\frac{d}{dt}\left(t^{\frac{1}{\xi}}y(t)\right),    
\end{gather*}
which yields, by integrating $t\in[s,s+\varepsilon]$ and then dividing both sides by $(s+\varepsilon)^{\frac{1}{\xi}}$, that
\begin{align*}
y(s+\varepsilon)&=\left(\frac{s}{s+\varepsilon}\right)^{\frac{1}{\xi}}y(s)
+\left[1-\left(\frac{s}{s+\varepsilon}\right)^{\frac{1}{\xi}}\right]\frac{1}{(s+\varepsilon)^{\frac{1}{\xi}}-s^{\frac{1}{\xi}}}\int\limits_s^{s+\varepsilon}\frac{1}{\xi}t^{\frac{1}{\xi}-1}\phi_y(t)\,dt\\
&=\left(\frac{s}{s+\varepsilon}\right)^{\frac{1}{\xi}}y(s)
+\left[1-\left(\frac{s}{s+\varepsilon}\right)^{\frac{1}{\xi}}\right]\frac{1}{(s+\varepsilon)^{\frac{1}{\xi}}-s^{\frac{1}{\xi}}}\int\limits_{s^{\frac{1}{\xi}}}^{(s+\varepsilon)^{\frac{1}{\xi}}}\phi_y(\tau^{\xi})\,d\tau.
\end{align*}
Using the convexity of the function $g$ and Jensen's inequality, the relation above gives
\begin{align*}
g(y(s+\varepsilon))&\leq\left(\frac{s}{s+\varepsilon}\right)^{\frac{1}{\xi}}g(y(s))
+\left[1-\left(\frac{s}{s+\varepsilon}\right)^{\frac{1}{\xi}}\right]\frac{1}{(s+\varepsilon)^{\frac{1}{\xi}}-s^{\frac{1}{\xi}}}\int\limits_{s^{\frac{1}{\xi}}}^{(s+\varepsilon)^{\frac{1}{\xi}}}g(\phi_y(\tau^{\xi}))\,d\tau\\
&\leq\left(\frac{s}{s+\varepsilon}\right)^{\frac{1}{\xi}}g(y(s))
+\left(\frac{1}{s+\varepsilon}\right)^{\frac{1}{\xi}}\int\limits_{s^{\frac{1}{\xi}}}^{(s+\varepsilon)^{\frac{1}{\xi}}}\frac{M_1}{\tau^{3\xi}}\,d\tau,    
\end{align*}
which implies, after multiplying both sides by $(s+\varepsilon)^{\frac{1}{\xi}}$ and dividing both sides by $\varepsilon$, that
\begin{gather*}
\frac{1}{\varepsilon}\left((s+\varepsilon)^{\frac{1}{\xi}}g(y(s+\varepsilon))-s^{\frac{1}{\xi}}g(y(s))\right)\leq\frac{1}{\varepsilon}\int\limits_{s^{\frac{1}{\xi}}}^{(s+\varepsilon)^{\frac{1}{\xi}}}\frac{M_1}{\tau^{3\xi}}\,d\tau.  
\end{gather*}
Letting $\varepsilon\to 0$, we get
\begin{gather*}
\frac{d}{ds}[s^{\frac{1}{\xi}}g(y(s))]\leq\frac{d}{d\varepsilon}\left(\int\limits_{s^{\frac{1}{\xi}}}^{(s+\varepsilon)^{\frac{1}{\xi}}}\frac{M_1}{\tau^{3\xi}}\,d\tau\right)\Bigg|_{\varepsilon=0}\,\,=\,\,\frac{M_1}{\xi}s^{\frac{1}{\xi}-4},
\end{gather*}
and then
\begin{gather*}
T^{\frac{1}{\xi}}g(y(T))-t_0^{\frac{1}{\xi}}g(y(t_0))=\int\limits_{t_0}^T\frac{d}{ds}\left[s^{\frac{1}{\xi}}g(y(s))\right]\leq\frac{M_1}{\xi\left(\frac{1}{\xi}-3\right)}\left(T^{\frac{1}{\xi}-3}-t_0^{\frac{1}{\xi}-3}\right).
\end{gather*}   
\end{proof}

\subsection{Proof of Proposition \ref{prop-202511021525}}
\begin{proof}
The key is to apply the Cauchy-Lipschitz-Picard theorem to the first-order dynamical system \eqref{202405051657} (see \cite[Proposition 6.2.1]{zbMATH00048893}).

For every $t\in\R$, we show that $T(t,\cdot)$ is Lipschitz continuous. Indeed, take arbitrarily $y=(y_1,y_2,y_3),z=(z_1,z_2,z_3)\in\calh\times\calh\times\calh$ and consider
\begin{align*}
&\quad\,\,\|T(t,y)-T(t,z)\|^2\\
&=\|y_2-z_2\|^2+\|y_3-z_3\|^2+\|\be_1(t)(y_2-z_2)+\be_2(t)(y_3-z_3)+\be_0(t)[U(\phi_y(t))-U(\phi_z(t))]\|^2\\
&\leq\|y-z\|^2+\left(\sum_{j=0}^2\be_j(t)^2\right)\left(\|y_2-z_2\|^2+\|y_3-z_3\|^2+\|U(\phi_y(t))-U(\phi_z(t))\|^2\right)\\
&\leq\|y-z\|^2+\left(\sum_{j=0}^2\be_j(t)^2\right)\left(\|y-z\|^2+L_u^2\|\phi_y(t)-\phi_z(t)\|^2\right).    
\end{align*}
Consequently, taking into account the explicit form of $\phi$ and the Cauchy-Schwarz inequality, we get
\begin{align*}
&\quad\,\,\|T(t,y)-T(t,z)\|^2\\
&\leq\|y-z\|^2+\left(\sum_{j=0}^2\be_j(t)^2\right)\left(\|y-z\|^2+L_u^2\left(1+\ld_1(t)^2+\ld_2(t)^2\right)\|y-z\|^2\right)\\
&=\left[1+\left(\sum_{j=0}^2\be_j(t)^2\right)\left(1+L_u^2\left(1+\ld_1(t)^2+\ld_2(t)^2\right)\right)\right]\|y-z\|^2.    
\end{align*}
Next is to prove $T(.,y)$ is locally integrable for every $y\in\calh\times\calh\times\calh$. Indeed, we observe
\begin{align*}
&\quad\,\,\int\limits_a^b\|T(t,y)\|^2\,dt-(b-a)\left(\|y_2\|^2+\|y_3\|^2\right)\\
&=\int\limits_a^b\|\be_1(t)y_2+\be_2(t)y_3+\be_0(t)[U(\phi_y(t))-U(x_*)]\|^2\,dt\\
&\leq\int\limits_a^b\left(\sum_{j=0}^2\be_j(t)^2\right)\left[\|y_2\|^2+\|y_3\|^2+L_u^2\|y_1-x_*+\ld_1(t)y_2+\ld_2(t)y_3\|^2\right]\,dt\\
&\leq\int\limits_a^b\left(\sum_{j=0}^2\be_j(t)^2\right)\left[\|y\|^2+L_u^2(2+\ld_1(t)^2+\ld_2(t)^2)(\|y\|^2+\|x_*\|^2)\right]\,dt.
\end{align*}
\end{proof}

\subsection{Proof of Proposition \ref{prop-20261018}}
\begin{proof}
The system \eqref{ode-1/t^3} can be written as a first-order system
\begin{gather*}
y'(t)=T(t,y(t)),
\end{gather*}
where $y=(y_1,y_2,y_3)$ and the mapping $T:\R\times\calh\times\calh\times\calh\to\calh\times\calh\times\calh$ is defined by
\begin{align*}
T(t,y)&:=\left(y_2,y_3,-\frac{\nu_2}{t}y_3-\frac{\nu_1}{t^2}y_2-\al_0\nabla f(y_1+\xi_1ty_2+\xi_2t^2y_3)\right).
\end{align*}
Thanks to the Cauchy-Lipschitz theorem and conditions of Proposition \ref{prop-20261018}, the first-order system admits a unique local solution. Let $x(\cdot)$ be the associated maximal solution defined on $[0, T_{\max})$, where $0 < T_{\max} \le\infty$. To establish $T_{\max} = \infty$, it suffices to prove the boundedness of $x(t),x'(t),x''(t)$. Indeed, we denote
\begin{align*}
u(t)&=\frac{\nu_2}{t}x''(t)+\frac{\nu_1}{t^2}x'(t),\\
v(t)&=(\xi_1+2\xi_2)tx''(t)+(1+\xi_1)x'(t),\\
g(t)&=\al_0\big(f(\phi_x(t)) - f_*\big) + \frac{1}{2}\Vert{}x''(t)\Vert{}^2 + \frac{1}{2}\Vert{}x'(t)\Vert{}^2.
\end{align*}
We have
\begin{align*}
\al_0\frac{d}{dt}f(\phi_x(t))
&=-\xi_2t^2\|x'''(t)\|^2-\inner{tx'''(t)}{\frac{1}{t}v(t)+\xi_2tu(t)}-\inner{u(t)}{v(t)}\\
&\leq-\frac{\xi_2}{2}t_0^2\|x'''(t)\|^2+\frac{1}{2\xi_2}\left\|\frac{1}{t}v(t)+\xi_2tu(t)\right\|^2+\|u(t)\|\cdot\|v(t)\|
\end{align*}
and
\begin{align*}
\frac{d}{dt}\left(\frac{1}{2}\Vert{}x''(t)\Vert{}^2 + \frac{1}{2}\Vert{}x'(t)\Vert{}^2\right) &= \langle x'''(t), x''(t) \rangle + \langle x''(t), x'(t) \rangle \\ &\le \frac{\xi_2 t_0^2}{4}\Vert{}x'''(t)\Vert{}^2 + \left(\frac{1}{\xi_2 t_0^2} + \frac{1}{2}\right)\Vert{}x''(t)\Vert{}^2 + \frac{1}{2}\Vert{}x'(t)\Vert{}^2.
\end{align*}
Now
\begin{align*}
g'(t)
&\leq\frac{1}{2\xi_2}\left\|\frac{1}{t}v(t)+\xi_2tu(t)\right\|^2+\|u(t)\|\cdot\|v(t)\|+ \left(\frac{1}{\xi_2 t_0^2} + \frac{1}{2}\right)\Vert{}x''(t)\Vert{}^2 + \frac{1}{2}\Vert{}x'(t)\Vert{}^2.
\end{align*}
Since $t_0 \le t \le T_{\max} \le \infty$, the quantities $\Vert u(t)\Vert$, $\Vert v(t)\Vert$, and $\left\Vert \frac{1}{t}v(t)+\xi_2tu(t)\right\Vert$ are linearly bounded in terms of $\Vert x'(t)\Vert$ and $\Vert x''(t)\Vert$. Therefore, there exists a constant $C_1 > 0$ (depending only on $\xi_1,\xi_2,\nu_1,\nu_2, t_0, T_{\max}$) such that
\begin{align*}
g'(t)\leq C_1\left(\frac{1}{2}\Vert{}x''(t)\Vert{}^2 + \frac{1}{2}\Vert{}x'(t)\Vert{}^2\right)\leq C_1g(t).    
\end{align*}
Applying Gr\"{o}nwall's inequality on $[t_0, T_{\max})$, we obtain
$$g(t) \le g(t_0) e^{C_1 (t - t_0)} \le g(t_0) e^{C_1 (T_{\max} - t_0)} := C_2 <\infty, \quad \forall t \in [t_0, T_{\max}).$$
From the definition of $g(t)$, it follows that $\Vert x'(t)\Vert \le \sqrt{2C_2}$ and $\Vert x''(t)\Vert \le \sqrt{2C_2}$ for all $t \in [t_0, T_{\max})$. Integrating over $[t_0, t] \subset [t_0, T_{\max})$ yields
$$\Vert x(t)\Vert \le \Vert x(t_0)\Vert + \int_{t_0}^t \Vert x'(s)\Vert\,ds \le \Vert x(t_0)\Vert + \sqrt{2C_2}(T_{\max} - t_0) <\infty.$$
Thus, $\Vert x(t)\Vert + \Vert x'(t)\Vert + \Vert x''(t)\Vert$ is uniformly bounded on $[t_0, T_{\max})$, which contradicts the blow-up criterion. Consequently, $T_{\max} = \infty$.
\end{proof}
\end{document}